\documentclass[a4paper, 12pt]{article}
\usepackage[left=3cm, right=2cm, top=2.5cm, bottom=2.5cm]{geometry}

\usepackage{amssymb,amsmath,amsthm,amsfonts,mathrsfs,bbm}
\usepackage[pdftex]{graphicx}
\usepackage{subcaption}
\usepackage{hyperref}
\usepackage{multirow}
\usepackage{cleveref}
\usepackage{listings}
\usepackage[T1]{fontenc}
\usepackage[utf8]{inputenc}
\usepackage[english]{babel}

\usepackage{csvsimple}
\usepackage{tikz}
\usepackage{pgfplots}
\pgfplotsset{compat=1.18}
\usepackage{pgfplotstable}

\usepackage{algpseudocode}

\usepackage{paralist,enumitem,tabularx}
\usepackage{mathtools}

\usepackage{layout}
\usepackage{fancyhdr}
\usepackage{setspace}

\usepackage{float}

\usepackage{pdfpages}

\usepackage{xcolor}

\usepackage{datetime}
\usepackage[backend=bibtex,maxnames=99]{biblatex} 
\numberwithin{equation}{section}

\theoremstyle{bolddef}
\newtheorem{definition}{Definition}[section]
\newtheorem{algorithm}[definition]{Algorithm}
\newtheorem{assumption}[definition]{Assumption}

\theoremstyle{boldplain}
\newtheorem{lemma}[definition]{Lemma}
\newtheorem{theorem}[definition]{Theorem}
\newtheorem{proposition}[definition]{Proposition}

\newcommand{\clball}{\overline{B}}

\newcommand{\xk}{x^k}
\newcommand{\yk}{y^k}
\newcommand{\wk}{w^k}

\newcommand{\tauk}{\tau_k}

\newcommand{\xkp}{x^{k+1}}

\newcommand{\xj}{x^j}
\newcommand{\xjp}{x^{j+1}}
\newcommand{\xjm}{x^{j-1}}

\newcommand{\pmin}{p_{\min}}

\newcommand{\taumin}{\tau_\mathrm{min}}
\newcommand{\taumax}{\tau_\mathrm{max}}

\newcommand{\R}{\mathbb{R}}
\newcommand{\barR}{\overline{\mathbb{R}}}
\newcommand{\N}{\mathbb{N}}
\newcommand{\X}{\mathbb{X}}
\newcommand{\dom}{\mathrm{dom}}
\newcommand{\Rvalue}{\mathcal{R}}
\newcommand{\dist}{\mathrm{dist}}

\newcommand{\proj}{\Pi}
\newcommand{\diag}{\operatorname{diag}}
\newcommand{\rank}{\operatorname{rank}}
\newcommand{\trace}{\operatorname{trace}}

\DeclareMathOperator*{\argmin}{arg\,min}

\date{\today}

\title{Projected Subgradient Methods for a Class of Nonsmooth and Nonconvex Optimization Problems}
\author{Christian Kanzow\footnote{University of W\"urzburg, Institute
	of Mathematics, Emil-Fischer-Str.\ 30, 97074 W\"urzburg, Germany; 
	e-mail: christian.kanzow@uni-wuerzburg.de} 
	\and Jannis Krüger \footnote{University of W\"urzburg, Institute
		of Mathematics, Emil-Fischer-Str.\ 30, 97074 W\"urzburg, Germany; 
		e-mail: jannis.krueger@uni-wuerzburg.de}
	\and Leo Lehmann\footnote{University of W\"urzburg, Institute
		of Mathematics, Emil-Fischer-Str.\ 30, 97074 W\"urzburg, Germany; 
		e-mail: leo.lehmann@uni-wuerzburg.de} }

\begin{document}

\maketitle

\begin{abstract}
	We investigate the optimization problem of minimizing a nonsmooth function that satisfies a nonsmooth version of the descent lemma over a nonempty and closed but not necessarily convex set. The objective function belongs to the class of upper-$\mathcal{C}^2$ functions, whereas the constraints may promote a sparse or low-rank structure. We propose a projected subgradient method with two different globalization strategies: (a) a nonmonotone linesearch and, under additional assumptions, (b) an auto-conditioned method, where the stepsize is given by a formula depending on data from past iterations. We show that both methods converge to solutions that satisfy a stronger stationarity concept than one would expect from the subdifferential sum-rule, which is particularly important since the optimization problems of interest are inherently nonconvex. Finally, we present promising numerical results when applying the algorithm to an MPEC-style problem as well as the matrix optimization problems MAXCUT and Robust PCA.
\end{abstract}

\par\addvspace{\baselineskip}

\par\addvspace{\baselineskip}

\section{Introduction}

Let us consider the optimization problem
\begin{equation}\label{eq:optproblem}
	\min_{x \in D} \varphi(x),
\end{equation}
where $D \subseteq \X$ is a nonempty and closed but not necessarily convex subset of some Euclidean space $\X$. We assume that the objective function $\varphi$ satisfies a nonsmooth (and local) formulation of the descent lemma. The corresponding class of functions is called upper-$\mathcal{C}^2$ and has been shown to have favorable properties for the application of subgradient methods with linesearch \cite{aragon2025}. Notably, every differentiable function with locally Lipschitz continuous gradient is upper-$\mathcal{C}^2$. On the other hand, every concave function is also upper-$\mathcal{C}^2$. Moreover, the sum of upper-$\mathcal{C}^2$ functions is still an upper-$\mathcal{C}^2$ function. Consequently, many DC (DC = difference-of-convex) programs belong to the class of problem \eqref{eq:optproblem}. This motivates our consideration of constrained optimization problems as in \eqref{eq:optproblem}.

In the differentiable case, it has recently been shown in \cite{olikier2025} that projected gradient methods with linesearch accumulate at Bouligand stationary points. In the case of a local Lipschitz gradient, the accumulation points are even proximally stationary, meaning that the negative gradient is a so-called proximal normal. Notably, the proximal normal cone is in general smaller than both the regular and limiting normal cone. We show that by applying a projected subgradient method with (nonmonotone) linesearch to our problem \eqref{eq:optproblem}, one obtains that limits of the subgradient sequence are contained in the proximal normal cone. This is a stronger stationarity than that obtained directly from calculus rules of the limiting subdifferential as the latter only guarantees the existence of a negative subgradient that is a limiting normal to $D$.

For a continuously differentiable function $ \varphi $, the projected gradient method is a standard method for optimization over a nonempty closed set $D$. It is based on the iteration
\begin{equation*}
	\xkp \in \proj_D(\xk - \tauk \nabla \varphi(\xk)),
\end{equation*}
where $\tauk$ is a stepsize and $\proj_D$ is the possibly set-valued projection onto $D$. In our descent method, where we allow for nonsmooth objective functions $\varphi$, we replace the gradient by an arbitrary element $\wk \in \partial \varphi(\xk)$, where $\partial \varphi$ denotes the limiting (or Clarke) subdifferential of $\varphi$.

The choice of $\tauk$ is crucial for the convergence behavior of the algorithm. Using a backtracking linesearch, the stepsize is chosen such that a suitable linesearch-criterion holds. However, by using the Armijo condition (see \cite{Armijo1966}) a descent in the function value is enforced in every iteration. In order to obtain (possibly) larger stepsizes and fewer backtracking iterations, we employ a nonmonotone method \cite{aragon2025, BirginMartinezRaydan2000, DeMarchi2023}. In the literature, there are essentially two approaches to obtain nonmonotone stepsize rules: The max-type rule by Grippo et al. \cite{GrippoLamparielloLucidi1986} and the mean-type rule by Zhang and Hager \cite{ZhangHager2004}. In the constrained case, both have been studied for differentiable objective functions in \cite{olikier2025}. On the other hand, without the projection step, they were also employed for our class of nonsmooth objectives in \cite{aragon2025} and \cite{kanzow2026}, respectively.

As a second contribution, under the stronger assumption that the nonsmooth formulation of the descent lemma holds with a global constant, we propose an auto-conditioned method, where the stepsize is calculated from an explicit formula depending on past iterates, function values and subgradients. This line of work has recently attracted a lot of interest as it may lead to a significant reduction in per-iteration cost compared the linesearch methods and still needs no problem-dependent parameters as input. A similar approach was proposed in the pioneering work \cite{malitsky2020} for convex and differentiable objectives, whereas our stepsize is closely related to the works \cite{lan2026} and \cite{yagishita2025}.

The paper is organized in the following way. In Section~\ref{Sec:Background}, we introduce some notation and concepts from variational analysis. We define the class of upper-$\mathcal{C}^2$ functions and recall some of their properties that make them compelling for optimization methods. Section~\ref{Sec:AlgorithmNL} presents our nonmonotone descent method to solve \eqref{eq:optproblem}. We discuss its properties and give (global) convergence guarantees. The subsequent Section~\ref{Sec:AlgorithmAC} introduces the auto-conditioned method for \eqref{eq:optproblem}. Again, we obtain essentially the same convergence results, however, under additional and stronger assumptions. Next, in Section~\ref{Sec:Numerics}, we apply the algorithm to do numerical experiments on three theoretically and practically relevant examples. Lastly, we conclude with some remarks in Section~\ref{Sec:Final}.

Notation: We will consider optimization problems on some nonempty closed set $D \subseteq \X$, where $\X$ is a Euclidean space (finite-dimensional Hilbert space). We identify the dual space $\X^*$ with $ \X $ itself. We write $ \langle x, y \rangle $ for the scalar product of two elements $ x, y \in \X $ and $ \| x \| $ for the induced norm of some $ x \in \X $. The projection of $x \in \X$ onto $D$ is denoted by $\proj_D(x) := \{ y \in \X \mid \dist(x, D) = \|x-y\| \}$, where the distance function is $\dist(x, D):=\inf_{y \in D} \| x-y\|$. If $D$ is, in addition, convex, $\proj_S(x)$ is single-valued for all $x \in \X$. Also, we write the indicator function of $D$ as $\delta_D$, which is $0$ for arguments in $D$ and $\infty$ else.

We denote by $B_r(x)$ the open and $\clball_r(x)$ the closed ball of some radius $r>0$ around $x \in \X$. For arbitrary $S \subseteq \X$, we denote its interior by $\mathrm{int}(S)$, the closure by $\mathrm{cl}(S)$, and the convex hull by $\mathrm{conv}(S)$.

Finally, let $ \R $ denote the set of real numbers, while $ \barR := (- \infty, 
+ \infty ] $ is the set of extended reals except that we exclude the
value $ - \infty $. Given an extended-valued function $ \theta: 
\X \to \barR $, we call $ \dom (\theta) := 
\{ x \in \X \mid \theta (x)  < \infty \} $
the domain of $ \theta $. The function $ \theta $ is said to
be proper if $ \dom (\theta) $ is nonempty. 

\section{Background Material}\label{Sec:Background}

Let us recall some results from variational analysis. We refer the
interested reader to the two excellent monographs \cite{Mordukhovich2018, RockafellarWets2009}
for more details.

Given a proper, lower semicontinuous function $ g: \X \to
\barR $ and any $ x \in \dom (g) $, we call 
\begin{equation*}
	\partial^F g(x) := \Big\{ v \in \X \, \Big|
	\liminf_{y \to x, y \neq x} \frac{g(y) - g(x) - 
	\langle v, y-x \rangle}{\| y- x \|} \geq 0 \Big\}
\end{equation*}
the \emph{regular} or \emph{Fr\'echet subdifferential} of $ f $ at $ x $, 
whereas
\begin{equation*}
	\partial^M g(x) := \big\{ v \in \X \mid 
	\exists \xk,  v^k \in \X : \xk \to x, g(\xk) \to g(x), v^k \to v,
	v^k \in \partial^F g(\xk) \ \forall k \big\}
\end{equation*}
is called the \emph{limiting, Mordukhovich}, or 
\emph{basic subdifferential}
of $ g $ at $ x $. 

For a function $g$ that is locally Lipschitz continuous around $\bar x \in \mathbb{R}$, we define the \emph{Clarke subdifferential} of $g$ at $\bar x $ as 
\begin{equation*}
	\partial^C g(\bar x):= \{ v \in \X : \langle v, h \rangle \leq g^\circ (\bar x; h) \text{ for all } h \in \X \},
\end{equation*}
where
\begin{equation*}
	g^\circ(\bar x; h) := {\lim \sup}_{x \to \bar x, t \to 0^+} \frac{g(x+th) - g(x)}{t}
\end{equation*}
is the \emph{Clarke directional derivative} of $g$ at $\bar x$ in the direction $h$. Later, for results that hold true for both the limiting and the Clarke subdifferential, we use the following convention:
\begin{equation*}
	\partial g \text{ means either } \partial^M g \text{ or } \partial^C g.
\end{equation*}
Let us recall the following properties: It holds that $\partial^C g(\bar x) = \mathrm{conv}(\partial^M g(\bar x))$ for all locally Lipschitz continuous functions. Further, in that case, the Clarke subdifferentials of $g$ on bounded subsets of $\X$ remain bounded, too, see e.g. \cite{clarke1998}. By consequence, the same is true for the limiting subdifferential. Both the Clarke and the limiting subdifferential share the following so-called \emph{robustness} property: Assume that $\{\xk\}$ is a sequence converging to some limit $\bar x$, and $\wk \in \partial g(\xk)$ for all $k \in \N$ converges to some $\bar w \in \X$. Then it holds that $\bar w \in \partial g(\bar x)$.

Let us next introduce the class of upper-$ \mathcal{C}^2 $ functions. These are central for our method and the convergence analysis thereof. For more details on this class of functions we refer to \cite{RockafellarWets2009}.

\begin{definition}\label{def:upperC2}
Let $ U \subseteq \X$ be an open set. We say that a function
$ \varphi: U \to \mathbb{R} $ is {\normalfont upper-$ \mathcal{C}^2 $} on $ U $, if, 
on some neighborhood $ V $ of each $ \bar{x} \in U $, there is a 
representation 
\begin{equation*}
	\varphi (x) = \min_{c \in C} \varphi_c (x),
\end{equation*}
where the functions $ \varphi_c $ are of class $ C^2 $ on $ V $, and 
$ C $ is a compact set (in some topological space) such that 
$ \varphi_c $ and its first- and second-order partial derivatives
depend continuously on $ (x,c) \in V \times C $.
\end{definition}

For example, taking the discrete topology, it follows that functions that are the pointwise minimum of a finite family of twice continuously differentiable functions $ f_i: U \to \mathbb{R} $ on some open set $ U \subseteq \X $ are upper-$\mathcal{C}^2 $. Further examples can be derived from the subsequent characterization of upper-$\mathcal{C}^2$ functions from the recent report \cite[Prop.\ 3.2]{aragon2025}. We also refer to the later Section~\ref{Sec:Numerics} for applications where these functions arise.

\begin{proposition}\label{Prop:CharUpperC2}
Let $ U \subseteq \X $ be an open set and $ \varphi: \X \to \mathbb{R} $ be locally Lipschitz on $ U $. Then the following statements
are equivalent:
\begin{itemize}
	\item[(a)] $ \varphi $ is upper-$\mathcal{C}^2 $ on $ U $.
	\item[(b)] For each $ \bar{x} \in U $, there exist a constant
	   $ \kappa \geq 0 $ and some neighborhood $ V $ of $ \bar{x} $ such 
	   that
	   \begin{equation}\label{eq:descentproperty}
	   	  \varphi (y) \leq \varphi (x) + \langle w, y-x \rangle + 
	   	  \kappa \| y - x \|^2
	   \end{equation}
	   for all $ x, y \in V $ and all $ w \in \partial \varphi (x) $.
 	\item[(c)] For each $ \bar{x} \in U $, there exists some neighborhood
 	   $ V $ of $ \bar{x} $ where $ \varphi $ can be expressed as
 	   $ \varphi = g - h $, where $ g $ is differentiable with Lipschitz
 	   gradient, and $ h $ is Lipschitz and prox-regular (see \cite[Definition 3.27]{RockafellarWets2009}). Indeed, one can 
 	   take $ g = \kappa \| \cdot \|^2 $, for some $ \kappa \geq 0 $, and 
 	   $ h $ to be convex.
\end{itemize}
\end{proposition}

We will later see that inequality \eqref{eq:descentproperty} from part (b) is crucial for our convergence analysis. Namely, it is a nonsmooth counterpart of the usual descent lemma for smooth functions with Lipschitz gradient. In this analogy, the constant $2 \kappa $ in \eqref{eq:descentproperty} corresponds to the Lipschitz constant. Note, however, that in the inequality above, $ \kappa $ is a local constant possibly depending on the given point $ \bar{x} $. Inequality \eqref{eq:descentproperty} holds for an arbitrary subgradient $ w \in \partial \varphi (x) $ and not just for a particular element from $ \partial \varphi (x) $. This is a very strong observation that is later important to show that our subgradient method is well-defined. 

Finally, the characterization in (c) links upper-$\mathcal{C}^2$ functions to the class of DC-functions, see the corresponding discussion in \cite{aragon2025}. More precisely, DC-functions where the first term is smooth are upper-$\mathcal{C}^2$.

We proceed by showing that if $\varphi$ is upper-$\mathcal{C}^2$ on $U$ and $D \subset U$ is a compact subset, then the constant $\kappa$ in characterization (b) can be chosen as a global constant. This result and its proof are essentially the same as for the corresponding fact that every locally Lipschitz continuous function is globally Lipschitz on compact sets. 

\begin{lemma}\label{lemma:upperc2compact}
	Let $\varphi: \X \to \R$ be upper-$\mathcal{C}^2$ on an open set $U \subseteq \X$ and let $C \subset U$ be compact. Then there exists a constant $\kappa_C$ such that
	   \begin{equation}\label{eq:globaldescentproperty}
	   	  \varphi (y) \leq \varphi (x) + \langle w, y-x \rangle + 
	   	  \kappa_C \| y - x \|^2
	   \end{equation}
	   holds for all $ x, y \in C $ and all $ w \in \partial \varphi (x) $.
\end{lemma}
\begin{proof}
As $\varphi$ is locally Lipschitz continuous, $\partial \varphi$ is locally bounded, hence, on the compact set $C$, there exists a constant $G$ satisfying $\|w \| \leq G$ for all $w \in \partial \varphi(x)$, for all $x \in C$. 

	As $\varphi$ is upper-$\mathcal{C}^2$, for every $\bar x \in C$, we find a constant $\kappa_{\bar x}$ and an open neighborhood $V_{\bar x}$, on which \eqref{eq:descentproperty} holds (with $\kappa = \kappa_{\bar x}$). Without loss of generality, we can assume $V_{\bar x} = B_{r_{\bar x}}(\bar x)$. Now, $\{B_{r_{\bar x}}(\bar x)\}_{\bar x \in C}$ is an open cover of $C$, and by compactness, we find a finite open subcover, $\{B_{r_i}(\bar x_i)\}_{i=1}^n$. Now, define
	\begin{equation*}
		M := \sup_{x, y \in C, w \in \partial \varphi(x)} \big\{\varphi(y) - \varphi(x) - \langle w, y -x\rangle\big\},
	\end{equation*}
	which is finite due to $\|w\| \leq G$ for all $w \in \partial \varphi(x)$, for all $x \in C$ and continuity in $x$ and $y$. Let $\kappa:= \max\{\kappa_{\bar x_1}, \dots, \kappa_{\bar x_n}, \frac{M}{r^2}\}$, where $r > 0$ is Lebesgue's number associated with the subcover and we assume $r< r_i$ for $i=1, \dots, n$.

	Let now $x, y\in C$ be arbitrary. Then there are two cases: First, assume $\|x-y\|\geq r$, then
	\begin{equation*}
		\varphi(y)-\varphi(x) - \langle w, y-x\rangle \leq M \leq \frac{M}{r^2} \| x- y\|^2 \leq \kappa \|x-y\|^2.
	\end{equation*}
	Secondly, if $\|x-y\| < r$, then, by definition of Lebesgue's number, there exists some $i \in \{1, \dots, n\}$, such that $x, y \in B_{r_i}(\bar x_i)$, hence 
	\begin{equation*}
		\varphi(y)-\varphi(x) - \langle w, y-x\rangle \leq \kappa_{\bar x_i} \| x- y\|^2 \leq \kappa \|x-y\|^2.
	\end{equation*}
	This completes the proof.
\end{proof}

If $C$ is a compact set, then we denote by $\kappa_C$ the constant such that \eqref{eq:globaldescentproperty} holds. For our auto-conditioned method later, we note that there are other cases where $\kappa$ in part (b) of Proposition~\ref{Prop:CharUpperC2} can be chosen as a global constant. For example, assume that $\varphi$ has a global representation as a minimum over a compact set, i.e. $\varphi(x) = \min_{c \in C} \varphi_c(x)$ for all $x \in U$ ($U \subseteq \X$ open) with $\varphi_c$ satisfying the properties as in Definition~\ref{def:upperC2}. Then, by \cite[Theorem 10.33]{RockafellarWets2009}, if $\rho$ is an upper bound to the Hessians $\| \nabla^2 \varphi_c(x)\|$ over $C \times U$, then $\varphi$ is upper-$\mathcal{C}^2$ with the (global) constant $\kappa = 2 \rho$.

The projection mapping has the following important properties that we cite from \cite{olikier2025}.

\begin{proposition}\label{prop:projproperties}
	Let $D \subset \X$ be nonempty and closed. Then for all $x \in D$, $v \in \X$, $y \in \proj_D(x-v)$, it holds that
	\begin{align}
		\|y-x\| \leq 2 \|v\|,\label{eq:projprop1}\\
		2 \langle v, y-x\rangle \leq - \| y-x \|^2,\label{eq:projprop2}
	\end{align}
	with strict inequalities in the case that $x \notin \proj_D(x-v)$.
\end{proposition}

Let us also recall some notation from variational geometry from \cite{RockafellarWets2009}. In the following $D \subset \X$ is assumed to be closed and nonempty. We define the \emph{tangent cone} of $D$ at $\bar x \in D$ as
\begin{equation*}
	T_D(\bar x) := \Big\{ d \in \X\, \Big|\, \exists \{ \xk \} \subset D, t_k \searrow 0, \xk \to \bar x, \frac{\xk - \bar x}{t_k} \to d\Big\}.
\end{equation*}
Further, the \emph{regular (or Fréchet) normal cone} is given by 
\begin{equation*}
	N_D^F(\bar x) = T_D(\bar x)^\circ := \Big\{ v \in X \, \Big|\, \langle v, d \rangle \leq 0\, \forall d \in T_D(\bar x)\Big\}.
\end{equation*}
and the \emph{limiting normal cone} is defined as 
\begin{equation*}
	N_D^M(\bar x) = \Big\{ v \in \X\, \Big|\, \exists \{\xk\} \subset C, \{ v^k\}: \xk \to \bar x, v^k \in N^F_D(\bar x), v^k \to v\Big\}.
\end{equation*}
Lastly, we define $N^P_D(\bar x)$ as the \emph{proximal normal cone} to $D$ at $\bar x$ as follows: 
\begin{equation*}
	N^P_D(\bar x) =\Big\{v \in \X\, \Big|\, \exists \bar \tau >0: \bar x \in \proj_D( \bar x+\bar \tau v)\Big\}.
\end{equation*}
Note that $v$ being proximal normal to $D$ at $x$ implies $\proj_D(x + \tau v) = \{x\}$ for all $\tau \in [0, \bar \tau)$. The proximal normal cone $N^P_D(\bar x)$ is a convex cone and the following inclusions hold and are in general strict (see, e.g. \cite{olikier2025}):
\begin{equation}\label{eq:normalinclusions}
	N_D^P(\bar x) \subseteq N_D^F(\bar x) \subseteq N_D^M(\bar x).
\end{equation}

Next, we introduce our notion of critical points. Let $\varphi$ be upper-$\mathcal{C}^2$ and let $D$ be a nonempty, closed set. We call $\bar x \in D$ a \emph{proximal critical point}, if it holds that
\begin{equation}\label{eq:proxcritical}
	0 \in \partial \varphi(\bar x) + N^P_D(\bar x).
\end{equation}
Equivalently, there exists some $w^* \in \partial \varphi(\bar x)$ such that $-w^* \in N^P_D(\bar x)$.

Below, in Proposition~\ref{prop:necoptcond}, we will show that a more restrictive version of criticality is a necessary condition for local minimizers of our optimization problem. Namely, if $\bar x$ is a local minimizer for \eqref{eq:optproblem}, then
\begin{equation*}
	- \partial \varphi(\bar x) \subseteq N^P_D(\bar x).
\end{equation*}

We note that, if $D$ is convex, both concepts are related to the respective terms in DC programming (see e.g.\ \cite{thi2026}) and, taking the results from Proposition~\ref{Prop:CharUpperC2} into account, one can establish a link between optimization problems for upper-$\mathcal{C}^2$ functions and DC programs.

Even the notion for critical points from \eqref{eq:proxcritical} here is stronger than usual necessary optimality conditions. In general, as $\varphi$ is locally Lipschitz continuous, by the sum-rule of the limiting subdifferential, one obtains $\partial (\varphi + \delta_D)(x) \subseteq \partial \varphi(x) + N_D^M(x)$, which together with the optimality condition $0 \in \partial (\varphi + \delta_D)(x)$ guarantees that there exists an element $w \in \partial \varphi(x)$, such that $-w \in N_D^M(x)$. In the upcoming proposition, however, by properties of the projection, the limiting normal cone $N_D^M(x)$ can be replaced by the smaller proximal normal cone $N_D^P(x)$. Secondly, the upper-$\mathcal{C}^2$ property guarantees that every negative subgradient is an element of the normal cone.

\begin{proposition}\label{prop:necoptcond}
	Under our usual assumptions, let $x \in D$ be a local minimum of \eqref{eq:optproblem}. Then for all $w \in \partial \varphi(x)$ it holds that $-w \in N^P_D(x)$.
\end{proposition}
\begin{proof}
	By contradiction, assume that there exists some $w \in \partial \varphi(x)$ such that $-w\notin N^P_D(x)$. Let $\rho \in (0, \infty)$. Then for all $\tau \in \left(0, \frac{\rho}{2 \|w\|} \right]$, it holds by \eqref{eq:projprop1} that
	\begin{equation*}
		x \notin \proj_D(x- \tau w) \subseteq B_{2 \tau \| w\|}(x) \subseteq B_\rho(x).
	\end{equation*}
	Hence, now for all $\tau \in \left( 0, \min\left\{ \frac{\rho}{2 \| w\|}, \frac{1}{2\kappa_{\clball_\rho(x)}}\right\}\right)$ and $y \in \proj_D(x-\tau w)$ (which implies $y \in B_\rho(x)$), we have by means of Lemma~\ref{lemma:upperc2compact} and \eqref{eq:projprop2} that
	\begin{equation*}
		\varphi(y) - \varphi(x) \leq \langle w, y-x\rangle + \kappa_{\clball_\rho(x)} \|y-x\|^2 < \left(\frac{-1}{2\tau} + \kappa_{\clball_\rho(x)}\right) \|y-x\|^2 \leq 0,
	\end{equation*}
	which contradicts the fact that $x$ was a local minimum.
\end{proof}

By the inclusions in \eqref{eq:normalinclusions}, we can also define weaker notions of criticality depending on the normal cone in use. This way, we define
\begin{itemize}
 \item \emph{F-criticality}: $0\in \partial \varphi(\bar{x}) + N^F_D(\bar{x})$.
 \item \emph{M-criticality}: $0\in \partial \varphi(\bar{x}) + N^M_D(\bar{x})$.
 \item \emph{C-criticality}: $0\in \partial \varphi(\bar{x}) + N^C_D(\bar{x})$, where the Clarke normal cone $N^C_D(\bar{x})$ is defined as the closed convex hull of $N^M_D(\bar{x})$.
\end{itemize}
These definitions are generalizations of the respective stationarity conditions for a smooth objective, cf.\ \cite{olikier2025}. The traditional way to define $*$-stationarity for $*\in \{F,M,C,P\}$ would be to have
\begin{align*}
	0\in \partial^{*}(\varphi+\delta_D)(\bar{x}),
\end{align*}
which is in general stronger than criticality as defined above. In our analysis we obtain results for P-criticality, as we utilize the robustness of the limiting subdifferential applied to the sequence of subgradients $w^k\in \partial \varphi(x^k)$ in our algorithm.

\section{Nonmonotone Linesearch Method}\label{Sec:AlgorithmNL}

The standard projected gradient method for the minimization of a continuously differentiable 
function $ \varphi $ over a closed nonempty set $D$ is based on the iteration
\begin{equation}\label{eq:opt-iteration}
	\xkp \in \proj_D(\xk - \tauk \nabla \varphi(\xk)),
\end{equation}
with a stepsize $ \tau_k $ usually chosen by a backtracking linesearch, i.e. by multiplying it with some $\beta \in (0,1)$ from an initial estimate in $[\taumin, \taumax]$ until a suitable linesearch criterion holds. The most common linesearch condition is the Armijo rule 
\begin{equation}\label{eq:Armijo}
	\varphi (\xkp) \leq \varphi (\xk) - \sigma \tau_k
	\| \nabla \varphi (\xk) \|^2
\end{equation}
for some constant $ \sigma \in (0,1) $. The descent method considered 
in this section is a direct generalization of this approach to the class
of objective functions $ \varphi $ which are upper-$\mathcal{C}^2 $. In the
monotone version, it is based on the iteration \eqref{eq:opt-iteration} for arbitrary elements $ \wk \in 
\partial \varphi (\xk) $ and $\yk \in \proj_D(\xk - \tauk \wk)$. The counterpart of the (monotone) Armijo rule 
\eqref{eq:Armijo} reads
\begin{equation}\label{eq:ArmijoNonmonotone}
	\varphi(\yk) \leq \varphi(\xk) + \sigma \langle \wk, \yk - \xk \rangle,
\end{equation}
for some $ \sigma \in (0,1)$.

The following result similar to \cite[Proposition 4.3]{aragon2025} and \cite{olikier2025} shows that this Armijo-type condition is satisfied for all sufficiently small stepsizes provided that $ \varphi $ is indeed upper-$\mathcal{C}^2$. Note that the interval in the definition of $\bar \rho$ is non-empty by boundedness of the subdifferential over compact sets.

\begin{proposition}\label{prop:welldef}
	Let $\bar x \in D$, $\taumax \in (0, \infty)$, $\sigma \in (0,1)$ and $\rho \in (0, \infty)$. Define
	\begin{align*}
		\bar \rho \in \left[ \rho + 2 \tau \sup_{x \in \clball_\rho(\bar x) \cap D, w \in \partial \varphi(x)} \|w\|, \infty \right) \text{ and } \tau^* = \frac{1-\sigma}{2 \kappa_{\clball_{\bar \rho}(\bar x) \cap D}}.
	\end{align*}
	Then, for all $x \in \clball_{\rho}(\bar x) \cap D$, $\tau \in [0, \min\{\tau^*, \taumax\}]$, $y \in \proj_D(x- \tau w), w \in \partial \varphi(x)$, it holds that
	\begin{equation}\label{eq:welldef}
		\varphi(y) \leq \varphi(x) + \sigma \langle w, y-x\rangle.
	\end{equation}
\end{proposition}

\begin{proof}
	For all $x \in \clball_{\rho}(\bar x) \cap D$, $\tau \in [0, \taumax]$, $y \in \proj_D(x- \tau w), w \in \partial \varphi(x)$, it holds by the triangle inequality and \eqref{eq:projprop1} that
	\begin{equation*}
		\| y- \bar x\| \leq \|y-x\|+\|x-\bar x\| \leq 2 \tau \| w\| + \rho \leq \bar \rho, 
	\end{equation*}
	which in turn implies $\proj_D(x - \tau w) \subseteq \clball_{\bar \rho}(\bar x)$. Consequently, we have for all $x \in \clball_{\rho}(\bar x) \cap D$, $\tau \in [0, \min\{\tau^*, \taumax\}]$, $y \in \proj_D(x- \tau w), w \in \partial \varphi(x)$, that by Lemma~\ref{lemma:upperc2compact} the following inequalities hold:
	\begin{align*}
		\varphi(y) &\leq \varphi(x) + \langle w, y-x\rangle + \kappa_{\clball_{\bar \rho}(\bar x)} \| y-x\|^2\\
		&\leq \varphi(x) + (1- 2 \tau \kappa_{\clball_{\bar \rho}(\bar x)} ) \langle w, y-x\rangle\\
		&\leq \varphi(x) + \sigma \langle w, y-x\rangle.
	\end{align*}
	The final inequality exploits the fact that $\langle w, y-x\rangle \leq 0$ in view of \eqref{eq:projprop2}.
\end{proof}

Note that in particular, the inner backtracking loop terminates with a stepsize contained in some compact interval
\begin{equation}\label{eq:intervaldef}
	\tau \in I := [\min\{ \tau_{\min}, \beta \tau^*\}, \taumax].
\end{equation}

We next provide a generalization to a nonmonotone version of the previous iteration. From \eqref{eq:welldef} we know that an algorithm with the iteration above and $\tauk$ determined by a backtracking linesearch with termination criterion as in \eqref{eq:welldef} will produce a monotonically decreasing sequence of function values $\{\varphi(\xk)\}_{k \in \mathbb{N}}$.

Now, \eqref{eq:welldef} will also hold if $\varphi(x)$ is replaced by an arbitrary upper bound. Therefore, we assume that we have reference values $\Rvalue_k \geq \varphi(\xk)$. Then Proposition \ref{prop:welldef} guarantees finite termination of a backtracking linesearch with the corresponding termination criterion
\begin{equation}\label{eq:linesearchtermination}
	\varphi(\yk) \leq \Rvalue_k + \sigma \langle \wk, \yk - \xk \rangle,
\end{equation}
with $\yk \in \proj_D(\xk - \tauk \wk)$. Hence, the linesearch for the stepsize parameter is well-defined for upper-$\mathcal{C}^2$ functions.

Crucially, the reason why nonmonotone methods may outperform their monotone counterparts in applications is that \eqref{eq:linesearchtermination} may allow for larger steps compared to a monotone method.

Before presenting the algorithmic details, let us discuss possible choices for the reference values $\Rvalue_k$.
One popular choice is due to Grippo et al.\ \cite{GrippoLamparielloLucidi1986}, where $ \Rvalue_k := \max \{ \varphi(\xj) \mid j = k, k-1, \ldots, k - m_k \} $ for some given (but bounded) sequence $ m_k \in \N $. As $\Rvalue_k$ is the maximum function value over the last few $m_k+1$ iterates, we refer to this strategy as the \emph{max-rule}. The max-rule for unconstrained minimization of upper-$\mathcal{C}^2$ functions was studied in \cite{aragon2025}. There, stationarity of accumulation points was derived.

In our Algorithm~\ref{Alg:NonmonotoneSubgradient}, we employ the reference values as introduced by Zhang and Hager \cite{ZhangHager2004}. Here, $ \Rvalue_{k+1} $ is computed as a convex combination of the previous reference value $ \Rvalue_k $ and the new function value $ \varphi (\xkp) $. This method is called the \emph{mean-rule}. We show in the upcoming result Lemma \ref{lemma:genproperties} that \eqref{eq:descentcondition} holds. Hence, the $\Rvalue_k$ chosen by the mean-rule are indeed an upper bound for $\varphi(\xk)$ in our setting. The algorithmic details are presented in Algorithm~\ref{Alg:NonmonotoneSubgradient}.

\begin{algorithm}[Nonmonotone Projected Subgradient Method]\leavevmode
	\label{Alg:NonmonotoneSubgradient}
%\caption{Nonmonotone Subgradient Method}
\begin{algorithmic}[1]
	\Require $x^0 \in C$, $0<\tau_\mathrm{min} \leq \taumax < \infty$, $\sigma, \beta \in (0,1)$, $\pmin \in (0, 1]$.
	\State Set $\Rvalue_0 := \varphi(x^0)$.
	\For{$k = 0, 1, 2, \dots$}
	\State Choose $\wk \in \partial \varphi(\xk)$; \If{suitable termination criterion holds} \State STOP and return $\xk$.\EndIf
	%\State Choose $\dk \in \mathbb{R}^n \setminus \{0\}$ such that $\langle \wk, \dk\rangle < 0$.
	\State Choose $\tau_k \in [\taumin, \taumax]$ and $\yk \in \proj_D(\xk - \tauk \wk)$
	\While{$\varphi(\yk) > \Rvalue_k + \sigma \langle \wk, \yk - \xk \rangle$}
	\State $\tauk = \beta \tauk$
	\State Choose $\yk \in \proj_D(\xk - \tauk \wk)$.
	\EndWhile
	\State Set $\xkp := \yk$.
	\State Choose $p_{k+1} \in [\pmin, 1]$ and set $\Rvalue_{k+1} := (1-p_{k+1})\Rvalue_k + p_{k+1} \varphi(\xkp)$.
	\EndFor
\end{algorithmic}
\end{algorithm}

Let us first collect some properties of the sequence generated by Algorithm~\ref{Alg:NonmonotoneSubgradient}. These are similar to those obtained for a nonmonotone proximal gradient method in \cite{DeMarchi2023}, where the same linesearch rules were used.

\begin{lemma}\label{lemma:genproperties}
Let $\varphi: \X \to \mathbb{R}$ be upper-$\mathcal{C}^2$ on $\X$ with $\inf \varphi > - \infty$. Then for all $x^0 \in \X$, Algorithm~\ref{Alg:NonmonotoneSubgradient} either stops at some stationary point after a finite number of iterations, or the sequence $\{\xk\}_{k \in \mathbb{N}}$ satisfies the following properties:
\begin{itemize}
	\item[(a)] For all $k \in \mathbb{N}$ it holds that
		\begin{equation}\label{eq:descentcondition}
			\varphi(\xkp) + (1- p_{k+1}) \delta_k \leq \Rvalue_{k+1} \leq \Rvalue_k - p_{k+1} \delta_k,
		\end{equation}
		where 
		\begin{equation}\label{eq:deltak}
			\delta_k := -\sigma \langle \wk , \xkp - \xk \rangle \geq \frac{\sigma}{2 \taumax} \|\xkp - \xk\|^2 \geq 0,
		\end{equation}
		where the inequality follows from \eqref{eq:projprop2}.
	\item [(b)] The sequence $\{ \Rvalue_k\}$ is monotonically decreasing.
	\item [(c)] Both $\{ \Rvalue_k \}$ and $\{ \varphi(\xk)\}$ converge to some value $\varphi^* \in \mathbb{R}$, i.e., both sequences converge and
	have the same limit.
	\item [(d)] $\|  \xkp - \xk \| \to 0$ for $k \to \infty$.
\end{itemize}
\end{lemma}

\begin{proof}
The first inequality of part (a) follows from the definition of $\Rvalue_{k+1}$ and the linesearch termination criterion in \eqref{eq:linesearchtermination}:
\begin{align*}
	\Rvalue_{k+1} &= (1-p_{k+1}) \Rvalue_k + p_{k+1} \varphi(\xkp) \\
	&\geq (1-p_{k+1}) \big(\varphi(\xkp) - \sigma \langle \wk , \xkp - \xk \rangle\big) + p_{k+1} \varphi(\xkp)\\
	&= \varphi(\xkp) - (1-p_{k+1})\sigma \langle \wk, \xkp - \xk \rangle.
\end{align*}
Similarly, we verify the second inequality:
\begin{align*}
	\Rvalue_{k+1} &= (1- p_{k+1}) \Rvalue_k + p_{k+1} \varphi(\xkp) \\
	&\leq (1-p_{k+1}) \Rvalue_k + p_{k+1} \big( \Rvalue_k + \sigma \langle \wk, \xkp - \xk \rangle \big)\\
	&= \Rvalue_k + p_{k+1} \sigma \langle \wk, \xkp - \xk \rangle.
\end{align*}
The second assertion directly follows from part (a) and we obtain for all $k$ that
\begin{equation*}
	\varphi(\xk) \leq \Rvalue_k \leq \dots \leq \Rvalue_0 = \varphi(x^0).
\end{equation*}
The last equation also shows that, as $\varphi$ is bounded from below, the sequence of reference values $\{\Rvalue_k\}$ is convergent to some limit $\varphi^*$. Now, as $p_k \geq \pmin$ and
\begin{equation*}
	\varphi(\xk) = \frac{1}{p_k} \big( \Rvalue_k - (1-p_k) \Rvalue_{k-1}
	\big) = \frac{1}{p_k} (\Rvalue_k - \Rvalue_{k-1}) + \Rvalue_{k-1}
\end{equation*}
due to the update of $ \Rvalue_k $,
the convergence of $\{\varphi(\xk)\}$ to the same limit $\varphi^*$ follows.

Statement (d) is now a consequence of part (a)
and the following telescoping argument:
\begin{equation}\label{eq:telescoperval}
	\infty > \Rvalue_0 - \varphi^* \geq \sum_{j=1}^k 
	\big( \Rvalue_{j-1} - \Rvalue_j \big) \geq \sum_{j=1}^k p_j \delta_{j-1} \geq \sum_{j=1}^k \pmin \frac{\sigma}{2 \taumax} \| \xj - \xjm \|^2,
\end{equation}
where we used \eqref{eq:descentcondition} and \eqref{eq:deltak}. The upper bound $\Rvalue_0 - \varphi^*$ does not depend on $k$ and therefore holds for all $k \in \mathbb{N}$.
\end{proof}

We note that \eqref{eq:telescoperval} in the proof above implies the following estimate similar to \cite[Proposition 4.4]{aragon2025}:

\begin{proposition}\label{prop:convest}
There exists a constant $c > 0$ such that
\begin{equation}
	\min_{0 \leq j \leq k} \| \xjp - \xj \| \leq \frac{c\sqrt{\Rvalue_0 - \varphi^*}}{\sqrt{k+1}}, \text{ for all } k \in \mathbb{N}.
\end{equation}
\end{proposition}
\begin{proof}
By \eqref{eq:telescoperval}, we obtain
\begin{equation*}
	\min_{0 \leq j \leq k} \| \xjp - \xj \|^2 \leq \frac{1}{k+1} \frac{2 \taumax}{\pmin \sigma}\sum_{j=1}^{k+1} \pmin \frac{\sigma}{2 \taumax} \| \xj - \xjm \|^2
	\leq \frac{1}{k+1} \frac{2 \taumax}{\pmin \sigma} \big(\Rvalue_0 - \varphi^* \big).
\end{equation*}
Taking the square root gives the claim with $c:= \sqrt{\frac{2 \taumax}{\pmin \sigma}}$.
\end{proof}

Next, we establish our global convergence result for the sequence $\{\xk\}$ generated by Algorithm~\ref{Alg:NonmonotoneSubgradient}. This result is inspired by \cite{olikier2025}, where the smooth case was considered.

\begin{theorem}\label{thm:globalconvergence}
	Let $\varphi: \X \to \mathbb{R}$ be upper-$\mathcal{C}^2$ with $\inf \varphi > - \infty$. Let $\{\xk\}$ be a sequence generated by Algorithm~\ref{Alg:NonmonotoneSubgradient}, given some $x^0 \in \X$. Then, either the method stops at a point with $-\wk \in N^P_D(\xk)$ after a finite number of iterations, or the following assertions hold for $\{\xk\}$:
	\begin{enumerate}
		\item Any accumulation point $x^*$ of the sequence of iterates is a proximal critical point. That is, there exists an $w^* \in \partial \varphi(x^*)$ such that $-w^* \in N^P_D(x^*)$.
		\item Let $x^*$ be an accumulation point of $\{\xk\}$. Then the sequence $ \{ \varphi (\xk) \} $ converges to $\varphi (x^*) $ and the sequence $ \{ \Rvalue_k \} $ converges monotonically to $ \varphi (x^*) $.
	\end{enumerate}
\end{theorem}
\begin{proof}
	Let $\{\xk\}_{k \in K}$ be a subsequence converging to $x^* \in C$. Let $\rho \in (0, \infty)$ and define $\bar \rho$, $\tau^*$ and the interval $I$ as before in \eqref{eq:intervaldef}. 
	For all $k \in K$ sufficiently large, it holds that $\xk \in \clball_{\rho}(x^*)$ and hence, we have $\xkp \in \proj_D(\xk - \tau^k \wk)$ for $\tau^k \in I$. This gives in particular
	\begin{equation}\label{eq:projdist1}
		\| \xkp - (\xk - \tau^k \wk ) \| = \dist(\xkp - \tau^k \wk, D).
	\end{equation}
	By taking a further subsequence, we can assume that $\tau^k$ converges to some $\tau>0$, as $I$ is compact and that $\{\wk\}$ converges as $\wk \to w^*$ by (local) boundedness of the subdifferential. By Lemma~\ref{lemma:genproperties} it holds that $\xkp$ converges to $x^*$, too. Now, taking the limit on the subsequence, we obtain from \eqref{eq:projdist1} that
	\begin{equation}\label{eq:stationarypoint1}
		\| x^* - (x^* - \tau w^* )\| = \dist(x^* - \tau w^*, D),
	\end{equation}
	by continuity of the distance function. This shows that $-w^* \in N^P_D(x^*)$. As $w^*$ is the limit of the subsequence $\wk \in \partial \varphi(\xk)$, possibly on a further subsequence, the robustness property of the limiting (or Clarke) subdifferential gives $w^* \in \partial \varphi(x^*)$. Together with $-w^* \in N^P_D(x^*)$, we get $0 \in \partial \varphi(x^*) + N^P_D(x^*)$.

	We already know that $\{\varphi(\xk)\}$ converges. However, on a subsequence, we have $\varphi(\xk) \to_K \varphi(x^*)$ by continuity of $\varphi$. Thus, the last claim concerning the convergence of $\{\varphi(\xk)\}$ and $\{\Rvalue_k\}$ follows.
\end{proof}

The proof of Theorem~\ref{thm:globalconvergence} reveals that the assertions above follow from two crucial properties: First, locally, the stepsize $\tauk$ stays in some compact interval $I \subseteq \mathbb{R}_+$ and secondly, it holds that $\|\xkp - \xk \| \to 0$. In the next section, under additional assumptions, we present a linesearch-free method that shares these properties.

\section{Auto-Conditioned Projected Subgradient Method}\label{Sec:AlgorithmAC}

In this section, we also propose an auto-conditioned projected subgradient method and derive corresponding convergence guarantees. The term auto-conditioned refers to the fact that, in particular, our method does not need any problem dependent parameters (such as the constant $\kappa$) and does not utilize any linesearch procedures. This type of stepsize strategies use a direct formula to compute the stepsize. Probably the most prominent example in the literature is the adaptive gradient method for convex optimization due to \cite{malitsky2020}. The method in \cite{malitsky2020} is fully adaptive and does not assume a global Lipschitz constant for the gradient. However, in the nonconvex case, corresponding results are harder to obtain. Our method is inspired by \cite{lan2026} and \cite{yagishita2025}, where the same stepsize strategy we use was employed in the context where $\varphi$ is a differentiable and globally Lipschitz-smooth function. The latter assumption seems to be necessary for nonconvex objectives. In the smooth setting, the common motivation behind these auto-conditioned algorithms is that their stepsizes are calculated using a formula to approximate (a multiple of) the inverse of the (local) Lipschitz constant of the gradient. Our method is presented in Algorithm~\ref{Alg:ACPSubgradient}. Note further that in contrast to the linesearch we introduced in the preceding section, the auto-conditioned stepsize is decreasing, making the method less adaptive to the local geometry. The main advantage, however, is that no possibly costly backtracking steps are employed. Algorithm~\ref{Alg:NonmonotoneSubgradient} features a two-loop structure, whereas for Algorithm~\ref{Alg:ACPSubgradient} a single loop is sufficient.

\begin{algorithm}[Auto-Conditioned Projected Subgradient Method]\leavevmode
	\label{Alg:ACPSubgradient}
%\caption{Nonmonotone Subgradient Method}
\begin{algorithmic}[1]
	\Require $x^0 \in C$, $\kappa_0 > 0$, $\alpha > 1$.
	\For{$k = 0, 1, 2, \dots$}
	\State Choose $\wk \in \partial \varphi(\xk)$; \If{suitable termination criterion holds} \State STOP and return $\xk$.\EndIf
	\State Set $\gamma_k = \max \{\kappa_0, \dots, \kappa_k\}$ and $\tauk = \frac{1}{2 \alpha \gamma_k}$.
	\State Choose $\xkp \in \proj_D(\xk - \tauk \wk)$.
	\State Set \begin{equation}\label{eq:defkappa}\kappa_{k+1} = \frac{\varphi(\xkp) - \varphi(\xk) - \langle \wk, \xkp - \xk\rangle}{\|\xkp - \xk \|^2}.\end{equation}
	\EndFor
\end{algorithmic}
\end{algorithm}

The formula \eqref{eq:defkappa} is motivated by part (b) in Proposition~\ref{Prop:CharUpperC2}, and provides an estimate for the constant $\kappa$. We will assume that Algorithm~\ref{Alg:ACPSubgradient} generates an infinite sequence of iterates. In particular, if $\|\xkp - \xk \| = 0$, then by \eqref{eq:proxcritical}, $\xk$ is already proximal critical and the algorithm can be stopped before $\kappa_{k+1}$ is calculated in \eqref{eq:defkappa}. 

In the motivating work \cite{yagishita2025}, the Algorithm~\ref{Alg:ACPSubgradient} was proposed for a proximal gradient method for composite problems, where one part of the objective is smooth. Let us note that the following property is independent of the choice of descent direction and hence carries over to our setting. First, observe that the sequence $\{\gamma_k\}$ is monotonically increasing and introduce the notation
\begin{equation*}
	S := \left\{k \in \N\, \Big|\, \frac{\alpha + 1}{2} \gamma_k \geq \kappa_{k+1} \right\} \quad \text{ and } \quad \bar S := \N \setminus S.
\end{equation*}
The following result is now due to \cite{yagishita2025}. For completeness, the proof is recalled here.

\begin{lemma}
	Let $\{\xk\}$ be the sequence generated by Algorithm~\ref{Alg:ACPSubgradient}. Then it holds that
	\begin{equation*}
		\frac{\alpha -1}{2} \sum_{j=0}^{k} \gamma_j \| x^{j+1} - x^j\|^2 \leq \varphi(x^0) - \varphi(\xkp) + \sum_{j \in \{0, \dots, k\} \cap \bar S} \left(\gamma_{j+1} - \gamma_j\right) \| x^{j+1} - x^j \|^2.
	\end{equation*}
\end{lemma}
\begin{proof}
	Recall that $\xkp \in \proj_D(\xk - \tauk \wk)$ is equivalent to
	\begin{equation*}
		\xkp \in \argmin_{y \in D} \| y - \xk + \tauk \wk \|^2 = \argmin_{y \in D} \frac{1}{2 \tauk} \| y - \xk \|^2 + \langle \wk, y - \xk \rangle.
	\end{equation*}
	Hence, we obtain from $\alpha \gamma_k = \frac{1}{2 \tauk}$ that
	\begin{equation*}
		\alpha \gamma_k \| \xkp - \xk \|^2 + \langle \wk, \xkp - \xk \rangle \leq 0,
	\end{equation*}
	by choosing $y = \xk$ in the equation above. On the other hand, the definition of $\kappa_{k+1}$ gives 
	\begin{equation*}
		\langle \wk, \xkp - \xk \rangle = \varphi(\xkp) - \varphi(\xk) - \kappa_{k+1} \| \xkp - \xk \|^2.
	\end{equation*}
	Together, these two equations yield
	\begin{equation}\label{eq:adaptivecombinedeq}
		\left(\alpha \gamma_k - \kappa_{k+1}\right) \| \xkp - \xk \|^2 + \varphi(\xkp) - \varphi(\xk) \leq 0.
	\end{equation}
	Now, we consider two cases: First, for $k \in S$, \eqref{eq:adaptivecombinedeq} yields
	\begin{equation}\label{eq:caseS}
	\begin{aligned}
		\varphi(\xk) - \varphi(\xkp) &\geq \left(\alpha \gamma_k - \kappa_{k+1}\right) \| \xkp - \xk\|^2 \\
			& \geq \left(\alpha \gamma_k - \left(\frac{\alpha+1}{2} \gamma_k\right)\right) \| \xkp - \xk\|^2 = \frac{\alpha -1}{2} \gamma_k \| \xkp - \xk\|^2.
	\end{aligned}
	\end{equation}
	In the second case, namely $k \in \bar S$, we have $\gamma_{k+1} = \max\{\kappa_0, \dots, \kappa_{k+1}\} = \kappa_{k+1}$ as $\alpha > 1$. Hence, applying \eqref{eq:adaptivecombinedeq} gives
	\begin{align*}
		(\alpha - 1) \gamma_k \| \xkp - \xk \|^2 &= \left( \alpha \gamma_k - \kappa_{k+1} + \gamma_{k+1} - \gamma_k \right) \| \xkp - \xk \|^2\\
		&\leq \varphi(\xk) - \varphi(\xkp) + \left(\gamma_{k+1} - \gamma_k\right) \| \xkp - \xk \|^2.
	\end{align*}
	Combining both cases and taking the sum, the assertion follows.
\end{proof}

However, in order to derive convergence guarantees for our auto-conditioned subgradient method, we require the following stronger assumption for the objective function $\varphi$, which guarantees that the stepsizes stay bounded away from zero.

\begin{assumption}\label{as:acpg}
	Assume $C \subseteq \X$ is nonempty and closed and that $\varphi$ is upper-$\mathcal{C}^2$ such that \eqref{eq:descentproperty} holds globally on $C$. Further, assume that $\inf \varphi > - \infty$.
\end{assumption}

Assumption~\ref{as:acpg} holds in particular if $D$ is a compact set, see Lemma~\ref{lemma:upperc2compact}. Further, under Assumption~\ref{as:acpg}, it holds that $\kappa_k \leq \kappa$ for all $k\geq 1$, where $\kappa$ is now a global constant on $D$. As claimed, we observe that this implies boundedness of the stepsizes $\tauk$, namely, we have that for all $k$ that
\begin{equation*}
	\tauk \in \left[ \frac{1}{2 \alpha \max\{\kappa_0, \kappa\}}, \frac{1}{2 \alpha \kappa_0}\right].
\end{equation*}

In \cite{yagishita2025} it was noted that the set $\bar S$ is finite. This can be seen as follows: For $k \in \bar S$ we have that $\frac{\alpha +1}{2} \gamma_k < \kappa_{k+1} = \gamma_{k+1} \leq \max\{\kappa_0, \kappa\}$. Assuming there is an infinite number of steps, where $k \in \bar S$, however, would mean that $\kappa_k \to \infty$, as $\frac{\alpha+1}{2}$ is a constant greater than $1$. This gives a contradiction.

Let us first derive an estimate similar to the corresponding result in Proposition~\ref{prop:convest} for our linesearch method.

\begin{proposition}\label{prop:convest2}
	There exists some $c > 0$ depending only on the iterates $x_0, \dots, x^{\max \bar S + 1}$ and in particular independent of $k$ such that
\begin{equation}
	\min_{0 \leq j \leq k} \| \xjp - \xj \| \leq \frac{c}{\sqrt{k+1}}, \text{ for all } k \in \mathbb{N}.
\end{equation}
\end{proposition}

\begin{proof}
	From the previous lemma and as $\kappa_0 \leq \gamma_j $, we have
	\begin{align*}
		\frac{(\alpha -1)(k+1) \kappa_0}{2} &\min_{j\in\{0, \dots k\}} \| x^{j+1} - x^j\|^2 
		\leq \frac{\alpha -1}{2} \sum_{j=0}^{k} \gamma_j \| x^{j+1} - x^j\|^2 \\
		&\leq \varphi(x^0) - \varphi(\xkp) + \sum_{j \in \{0, \dots, k\} \cap \bar S} \left(\gamma_{j+1} - \gamma_j\right) \| x^{j+1} - x^j \|^2,
	\end{align*}
	and $\varphi(x^0) - \varphi(\xkp) \leq \varphi(x^0) - \varphi^*$, where $\varphi^* = \inf \varphi$.

	Due to the finiteness of $\bar S$, the sum in the last term can be bounded by a finite value independent of $k$. The claim follows by rearranging terms.
\end{proof}

Further, Algorithm~\ref{Alg:ACPSubgradient} has the following properties, which are direct consequences of $\bar S$ being finite.

\begin{proposition}
	Under Assumption~\ref{as:acpg}, the following assertions hold:
	\begin{enumerate}
		\item The sequence $\{\varphi(\xk)\}$ converges to a finite value.
		\item It holds that $\|\xkp - \xk \| \to 0$.
	\end{enumerate}
\end{proposition}

\begin{proof}
As $\bar S$ is finite, for all $k$ large enough, \eqref{eq:caseS} holds. Hence, 
	\begin{equation*}
		0 \leq \frac{\alpha -1}{2} \kappa_0 \| \xkp - \xk \|^2 \leq \varphi(\xk) - \varphi(\xkp).
	\end{equation*}
	Both claims now follow directly by boundedness of $\varphi$ from below.
\end{proof}

Finally, this allows us to obtain a global convergence result as before.

\begin{theorem}\label{thm:globalconvergence2}
	Let Assumption~\ref{as:acpg} hold. Then, given any $x^0 \in \X$, the results from Theorem~\ref{thm:globalconvergence} hold for the sequence generated by the method in Algorithm~\ref{Alg:ACPSubgradient} (except for the convergence of $\{\Rvalue_k\}$, which is not available in our setting).
\end{theorem}
\begin{proof}
	The proof is the same as for Theorem~\ref{thm:globalconvergence}.
\end{proof}

\section{Numerical Results}\label{Sec:Numerics}

We now want to evaluate the algorithmic performance of the nonmonotone projected subgradient method from Algorithm~\ref{Alg:NonmonotoneSubgradient} as well as the auto-conditioned one from Algorithm~\ref{Alg:ACPSubgradient}. We implemented both algorithms in Python and selected three insightful and relevant examples for testing. In both implementations we used a stopping criterion modeled after the optimality condition \eqref{eq:proxcritical}, namely terminating if
\begin{equation}\label{eq:terminationcrit}
 \|x^k-\proj_D(x^k-\tau w^k)\|_{\infty}\leq \varepsilon
\end{equation}
for small $\varepsilon,\tau>0$, where $w^k\in \partial \varphi(x^k)$. Note that if \eqref{eq:terminationcrit} holds with $\varepsilon = 0$, then the iterate $\xk$ is indeed already a proximal critical point.

\subsection{MPEC-style model problem}

Mathematical programs with equilibrium/complementarity constraints (MPECs or MPCCs) are notoriously difficult to optimize with standard methods, as their feasible set may contain points that do not fulfill any common constraint qualification \cite{scheelScholtes2000}. For this problem class, one can construct examples, for which the different normal cones introduced in Section~\ref{Sec:Background} do not coincide, and thus algorithms which are searching for points with weaker stationarity may converge to suboptimal solutions. In the following example, we can identify three points with different stationarity properties, hence we can illustrate the advantageous theoretical guarantees of our projected subgradient method in this case.

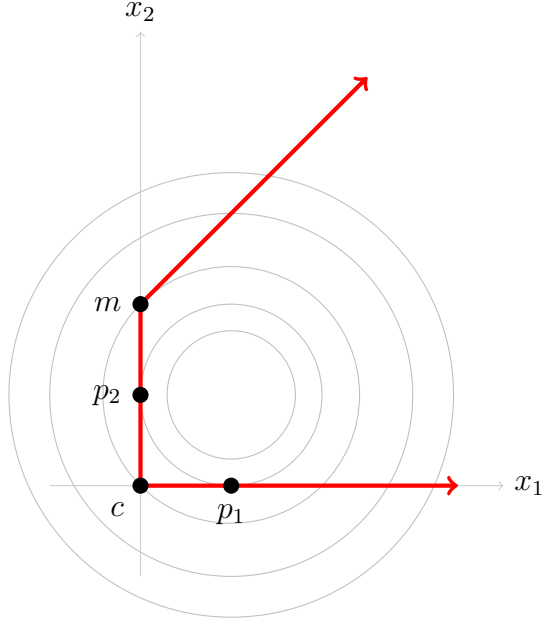
\begin{figure}[htbp]
    \centering
    \begin{tikzpicture}[
        scale=1.2,
        dot/.style={
            circle,
            draw,
            fill=black,
            inner sep=2pt
        }
    ]

        % Coordinate axes
        \draw[gray!40, ->]
            (-1,0) -- (4,0)
            node[right, black] {$x_1$};

        \draw[gray!40, ->]
            (0,-1) -- (0,5)
            node[above, black] {$x_2$};

        % ------------------------------------------------
        % Contours of f(x_1,x_2)
        %
        % f = 1/2 (x_1-1)^2 + 1/2 (x_2-1)^2
        % f = alpha  <=>  radius = sqrt(2 alpha)
        % ------------------------------------------------

        \draw[gray!50, thin]
            (1,1) circle ({sqrt(0.5)});

        \draw[gray!50, thin]
            (1,1) circle (1);

        \draw[gray!50, thin]
            (1,1) circle ({sqrt(2)});

        \draw[gray!50, thin]
            (1,1) circle (2);

        \draw[gray!50, thin]
            (1,1) circle ({sqrt(6)});

        % ------------------------------------------------
        % Feasible set
        % ------------------------------------------------

        % [0,infty) x {0}
        \draw[ultra thick, red, ->]
            (0,0) -- (3.5,0);

        % {0} x [0,2]
        \draw[ultra thick, red]
            (0,0) -- (0,2);

        % {(t,t+2) : t >= 0}
        \draw[ultra thick, red, ->]
            (0,2) -- (2.5,4.5);

        % ------------------------------------------------
        % Distinguished points
        % ------------------------------------------------

        \node[dot, label=left:{$p_2$}]
            at (0,1) {};

        \node[dot, label=below:{$p_1$}]
            at (1,0) {};

        \node[dot, label=left:{$m$}]
            at (0,2) {};

        \node[dot, label=below left:{$c$}]
            at (0,0) {};

    \end{tikzpicture}
    \caption{Feasible set of the MPEC-style problem with contours of the objective.}
    \label{fig:MPECfeas}
\end{figure}

To this end, we consider the following two-dimensional optimization problem
\begin{equation}
 \min_{(x_1,x_2)\in \R^2}f(x_1,x_2):=\frac{1}{2}(x_1-1)^2+\frac{1}{2}(x_2-1)^2 \quad \text{s.t.}\quad x\in D,
\end{equation}
where the feasible set $D$ is defined as
\begin{align*}
 D:=\Big([0,\infty) \times \{0\}\Big) \cup \Big(\{0\}\times [0,2]\Big) \cup \Big\{(t,t+2) \, \Big| \, t\geq 0\Big\},
\end{align*}
cf.\ Figure~\ref{fig:MPECfeas}. This example is a modification of the MPEC used in \cite[Example 10]{scheelScholtes2000}. The feasible set is obviously non-empty and closed, also by the construction as the union of three closed, convex, non-empty sets, the projection of $x$ onto $D$ can be easily calculated by first projecting onto each of the three line segments, and then just comparing the distances. The projections onto the line segments with minimal distance are then the set members of $\proj_D(x)$.

The objective function $f$ describes a two-dimensional paraboloid centered around $(1,1)$. In particular, the objective function is smooth, hence the notions of criticality as defined in Section~\ref{Sec:Background} and stationarity, as defined for example in \cite{olikier2025}, coincide. Thus, in the following we use the more commonly used term $*$-stationarity to describe the optimiality conditions
\begin{align*}
 0\in \nabla f(\bar{x})+N^*_D(\bar{x})
\end{align*}
for $*\in\{P,F,M,C\}$.

By the monotonicity of the objective, we see that $f$ attains its minimal value on $D$ in the closest points to $(1,1)$, which are $p_1:=(1,0)$ and $p_2:=(0,1)$. By Proposition~\ref{prop:necoptcond}, both points must be P-stationary, and hence also F- and M-stationary. The point $c:=(0,0)$ is non-optimal, but C-stationary, which holds analogously to \cite[Example 10]{scheelScholtes2000}. In fact, Scheel and Scholtes showed the stricter property that $c$ is even C-stationary in the MPEC sense. We now show that the point $m:=(0,2)$ is M-stationary, but not F-stationary, hence also suboptimal.

We have
\begin{align*}
 \nabla f(m)= \begin{pmatrix}
                 -1 \\ 1
                \end{pmatrix},
\end{align*}
hence we want to show that $(1,-1)$ lies in $N^M_D(m)$. We define the sequence $x^k:=(1/k,2+1/k)\to m$. In order to determine $N^F_D(x^k)$, one can see that, as each $x^k$ lies on the line segment $\{(t,t+2)\mid t>0\}$, the tangent cone consists only of the direction of the line segment, namely $T_D(x^k)=\{(t,t)\mid t\in \R\}$. We have for the polar cone
\begin{align*}
 N^F_D(x^k)=T_D(x^k)^\circ = \{(t,-t)\mid t\in \R\}.
\end{align*}
In particular, this shows $(1,-1)\in N^F_D(x^k)$ for each $k\in \N$. By defining $v^k\equiv (1,-1)$, this shows
\begin{align*}
 -\nabla f(m)\in N^M_D(m)
\end{align*}
and thus the M-stationarity of $m$. However, we have $d:=(0,-1)\in T_D(m)$, and thus
\begin{align*}
 -\nabla f(m)^T d = 2>0,
\end{align*}
which shows that $m$ is not F-stationary.

This problem is an example for a setting, where using a method searching only for M- or C-stationary points may lead to suboptimal results, and the theoretical guarantees provided in Sections~\ref{Sec:AlgorithmNL} and~\ref{Sec:AlgorithmAC} matter. For numerical confirmation, we ran both projected gradient methods on a grid of $10000$ starting values equidistantly spaced in the box $[-1,4]^2$. We set $\varepsilon = 10^{-6}$ and $\tau = 0.1$. In all instances, the projected gradient method converges to one of the optimal P-stationary points $p_1$ or $p_2$, as displayed in Table~\ref{tab:attractorsMPEC} and visualized in Figure~\ref{fig:basins}.

\begin{table}
\centering
\begin{tabular}{c|cc}
 Method & LS & AC \\
 \hline
 converging to $p_1$ & 4687 & 4864 \\
 converging to $p_2$ & 5313 & 5136
\end{tabular}
\caption{Convergence of starting grid values for the nonmonotone linesearch (LS) and the auto-conditioned stepsize (AC)}
\label{tab:attractorsMPEC}
\end{table}

\begin{figure}
\begin{subfigure}{.5\textwidth}
 \centering
  \includegraphics[width=0.9\linewidth]{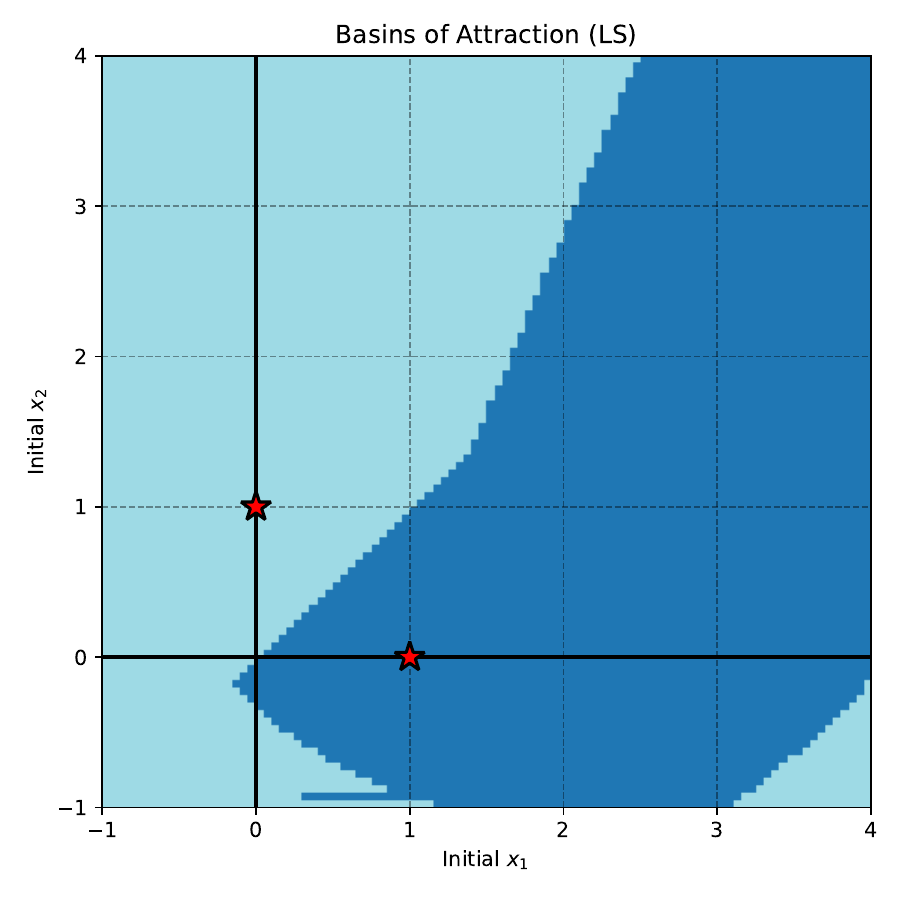}
\end{subfigure}
\begin{subfigure}{.5\textwidth}
 \centering
  \includegraphics[width=0.9\linewidth]{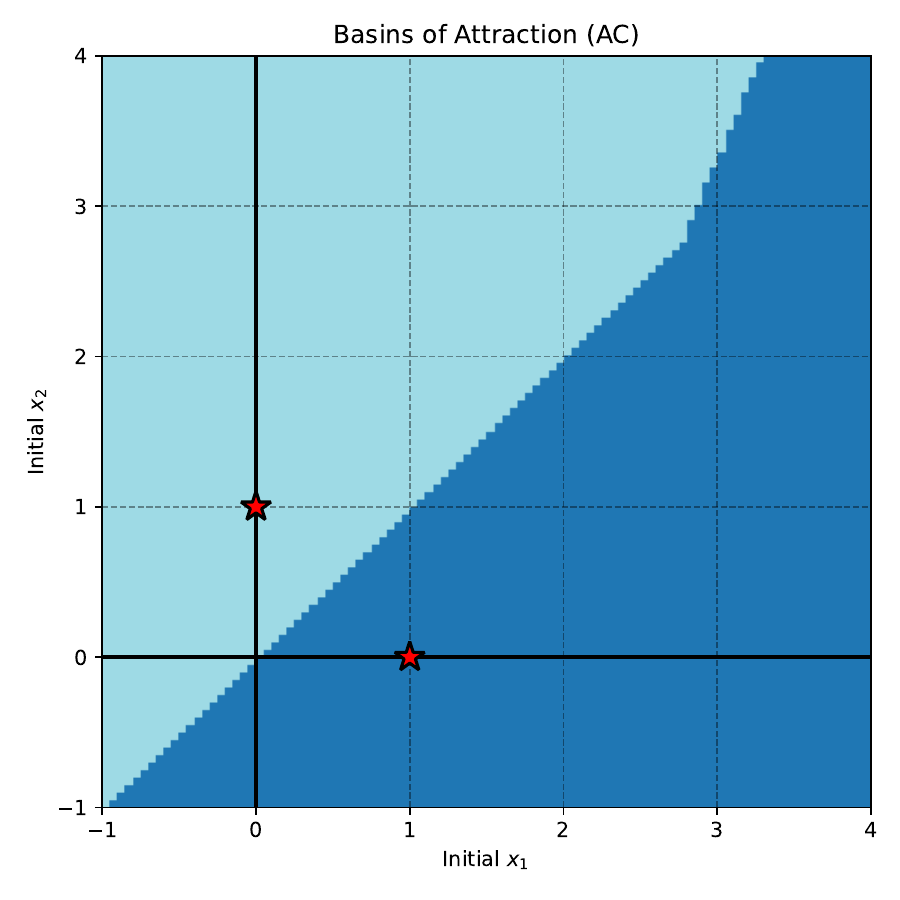}
\end{subfigure}
  \caption{Basins of attraction for the grid of starting values. Dark area: convergence to $p_1$, light area: convergence to $p_2$}
  \label{fig:basins}
\end{figure}

\subsection{MAXCUT}

The MAXCUT problem is an NP-hard \cite{GoemansWilliamson} graph theoretical problem, we follow the definition from \cite{KanzowGeometric}: Let $G:=(V,E)$ be a weighted undirected graph with vertex set $V=\{1,\ldots, n\}$ and edges $e_{ij}$ connecting $i,j\in V$. Let $A=(a_{ij})$ be the (symmetric) adjacency matrix, where an entry $a_{ij}=a_{ji}$ denotes the non-negative weight of $e_{ij}$, which can also be set to zero. We now define a \emph{cut} for $S\subseteq V$ as the set
\begin{align*}
 \delta(S):=\{e_{ij}\mid i\in S,j\in \bar S\},
\end{align*}
with $\bar S = V\setminus S$. This describes the set of all edges which connect a point in $S$ to a point in $\bar S$. The weight of the cut is defined by
\begin{align*}
 w(S):=\sum_{e_{ij}\in \delta(S)}a_{ij}.
\end{align*}
Our goal is now to find the cut with maximum  weight. Following the reasoning of \cite{GoemansWilliamson}, the problem is transformed to a semidefinite program by assigning a value $w_i\in\{-1,+1\}$ for each vertex $v_1,\ldots, v_n$, where $w_i=1$ if $v_i\in S$ and $w_i=-1$ if $v_i\in \bar S$. Now, instead of having $w_i$ with $|w_i|=1$ be elements of the unit sphere in one dimension, the problem is generalized by considering vectors $y_i\in \R^n$ with $\|y_i\|_2=1$. Defining the matrix $Y\in \R^{n\times n}$ as $(y_{ij})$ with $y_{ij}=y_i^Ty_j$, we obtain the semidefinite relaxation to the MAXCUT problem via $L:=\diag(Ae)-A$ by
\begin{equation}\label{prob:SDR}
 \min_{Y\in \R^{n\times n}_{\mathrm{sym}}} -\frac{1}{4}\trace(LY) \quad \text{s.t.}\quad \diag Y = e,\, Y\succeq 0.
\end{equation}
The solution of this problem is matrix-valued and has to be transformed back in order to obtain a cut vector, for details see \cite{GoemansWilliamson}. However, if we add the additional constraint $\rank Y=1$ to problem \eqref{prob:SDR}, a feasible Matrix $Y$ is a symmetric matrix of rank one with $\diag Y=e$, which results in $Y=ww^T$ with $w\in \{-1,1\}^n$. Hence by adding the rank constraint and determining the vector $w$, we recuperate a solution of the original MAXCUT problem.
Our resulting optimization problem is now
\begin{equation}\label{prob:maxcutRank}
 \min_{W\in \R^{n\times n}_{\mathrm{sym}}} -\frac{1}{4}\trace(LW) \quad \text{s.t.}\quad \diag W = e,\, W\succeq 0,\, \rank W =1.
\end{equation}
As the feasible set is still not tractable for projection, we modify the problem to
\begin{equation}\label{prob:maxcutPG}
 \min_{W\in \R^{n\times n}_{\mathrm{sym}}} -\frac{1}{4}\trace(LW)+\rho \Big(\trace (W) -\|W\|_2\Big) \quad \text{s.t.}\quad \diag W = e,\, W\succeq 0
\end{equation}
with a penalty parameter $\rho > 0$. The penalty term is constructed utilizing several properties of the problem. Let $\lambda_1(W)\geq \lambda_2(W)\geq\ldots\geq \lambda_n(W)\geq 0$ be the eigenvalues of the symmetric positive semidefinite matrix $W$ sorted by value from largest to smallest.
\begin{itemize}
 \item By symmetry and positive semidefiniteness we have $\|W\|_2=\lambda_1(W)$ and
\begin{align*}
 \trace(W)=\sum_{i=1}^n \lambda_i(W)\geq 0
 \end{align*}
 Hence we obtain the expression
 \begin{align*}
  \trace(W)-\|W\|_2=\sum_{i=2}^n \lambda_i(W)
 \end{align*}
 and we penalize the $n-1$ smallest eigenvalues of $W$ in order to obtain a matrix of rank $\leq 1$.
 \item By the diagonal constraint $\diag W = e$, the zero-matrix is not feasible, and the penalty term is a sensible choice for enforcing matrix rank $1$.
\end{itemize}

Next we want to certify that the problem formulation \eqref{prob:maxcutPG} fits into our projected gradient framework. First, owing to Proposition~\ref{Prop:CharUpperC2} (c), we see that our objective function is upper-$\mathcal{C}^2$ by writing
\begin{align*}
 g(W):=\trace\left(\left(-\frac{1}{4}L+\rho I\right)W\right), \quad h(W)=\rho\|W\|_2
\end{align*}
and noting that $g$ as a linear function is differentiable with Lipschitz gradient and $h$ as a scaled norm is convex. Furthermore, by setting the objective $f=g-h$, the limiting subgradient is given by
\begin{align*}
 \partial f(W)=-\frac{1}{4}L+\rho I -\rho\Big\{ vv^T\mid v\in E_{\lambda_1(W)},\|v\|=1\Big\},
\end{align*}
where $E_{\lambda_1(W)}$ denotes the eigenspace of the largest eigenvalue of $W$.

Projecting onto the feasible set
\begin{align*}
 D:=\Big\{W\in\R^{n\times n}\mid W\text{ symmetric},W\succeq 0,\diag W=e \Big\}
\end{align*}
has no explicit form, but it is a common problem known as \emph{Nearest Correlation Matrix}, and there exist numerically efficient ways to compute this projection. In our implementation, we used a method by Borsdorf and Higham, which in turn is a preconditioned form of a semismooth Newton method, cf.\ \cite{BorsdorfNearCorr}.

As a test problem collection, we used the \texttt{rudy} dataset from the MAXCUT problems in the Biq Mac library\footnote{\url{https://biqmac.aau.at/biqmaclib.html}}, which consists of $130$ randomly generated graphs. The optimal values of the problems are known. In order to obtain a good starting value for our projected subgradient method, we ran the \texttt{SCS} solver for semidefinite programs from the \texttt{cvxpy} package \cite{cvxpy} on the semidefinite convex relaxation \eqref{prob:SDR}, and fed the output into our algorithms. We ran all problem instances with a tolerance of $\varepsilon=10^{-6}$ and chose $\tau=0.1$ as well as $\rho=5$.

Algorithms~\ref{Alg:NonmonotoneSubgradient} and~\ref{Alg:ACPSubgradient} terminated successfully on all problem instances. The solutions of the two algorithms coincide on all problems and achieve a lower relative error than the solution of the semidefinite relaxation. A more detailed evaluation of the accuracies of our PG implementation versus the semidefinite relaxation is given in Figure~\ref{fig:accuracy-counts-rudy}.

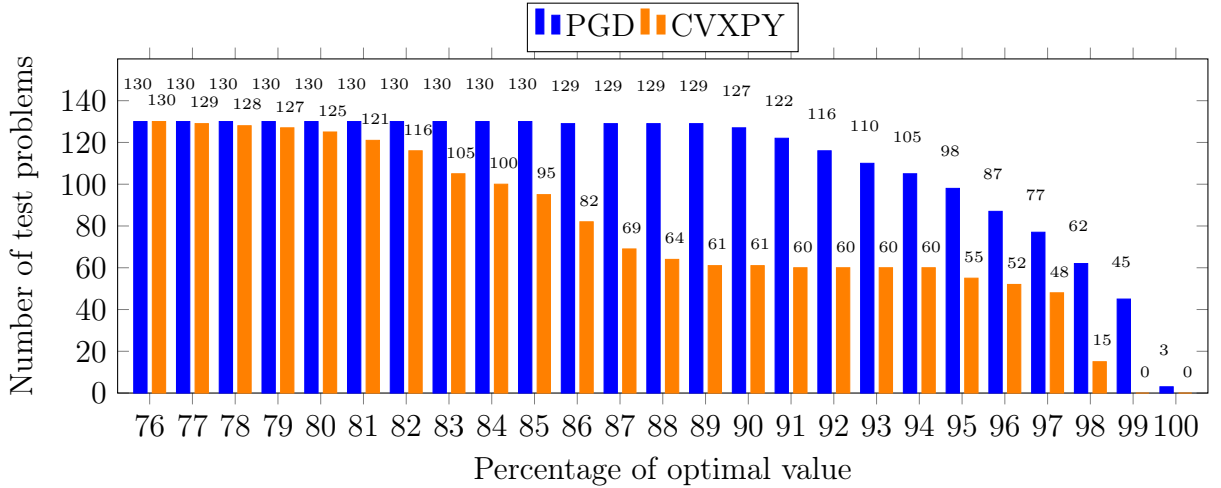
\begin{figure}[ht]
    \centering
    \begin{tikzpicture}
        \begin{axis}[
            width=\textwidth,
            height=6cm,
            ybar,
            bar width=5pt,
            xlabel={Percentage of optimal value},
            ylabel={Number of test problems},
            xmin=75.5,
            xmax=100.5,
            ymin=0,
            ymax=160,
            xtick={76,...,100},
            ytick={0,20,...,140},
            enlarge x limits=0.01,
            legend style={
                at={(0.5,1.02)},
                anchor=south,
                legend columns=-1
            },
            nodes near coords,
            every node near coord/.append style={
                font=\tiny
            },
        ]

        \addplot[
            fill=blue,
            draw=blue,
            nodes near coords,
            every node near coord/.append style={
                font=\tiny,
                xshift=-1pt,
                yshift=8pt,
            },
        ]
        table[
            x=x,
            y=count_pgd,
            col sep=comma
        ]{accuracy_counts.csv};

        \addplot[
            fill=orange,
            draw=orange,
            nodes near coords,
            every node near coord/.append style={
                font=\tiny,
                xshift=1pt,
                yshift=2pt,
            },
        ]
        table[
            x=x,
            y=count_cvxpy,
            col sep=comma
        ]{accuracy_counts.csv};

        \legend{PGD, CVXPY}

        \end{axis}
    \end{tikzpicture}
    \caption{Accuracies for the MAXCUT problems from the \texttt{rudy} dataset in the Biq Mac library.}
    \label{fig:accuracy-counts-rudy}
\end{figure}

Here we can also see the superior performance of PGD on the \texttt{rudy} collection. For example, our method is capable of solving $62$ of the $130$ problems with an accuracy of at least $98\%$, where the semidefinite relaxation achieves this accuracy only on $15$ problems.
We also note that, as in the reasoning above, the results of the semidefinite relaxation cannot be directly interpreted as a solution to MAXCUT, we observed them to be even infeasible to problem \eqref{prob:maxcutRank} as they generally have full rank.

Comparing both PG approaches, we display the metrics of CPU time, iterations of the projected subgradient method and of the Newton iterations needed for calculating the projections in Figure~\ref{fig:performance_profiles}.
\begin{figure}
    \centering

    % ============================================================
    % Row 1: CPU time
    % ============================================================

    \begin{subfigure}{0.49\textwidth}
        \centering
        \begin{tikzpicture}
            \begin{axis}[
                width=\linewidth,
                ylabel={Number of problems solved},
                %xmode=log,
                grid=major,
                legend pos=south east,
                legend style={font=\tiny},
                ymajorgrids=true,
                xmajorgrids=true,
            ]

            \addplot[
                thick,
                const plot
            ] table[
                x=tau,
                y=rho,
                col sep=comma
            ] {ppLScpu_time_seconds.csv};
            \addlegendentry{LS}

            \addplot[
                red,
                thick,
                const plot
            ] table[
                x=tau,
                y=rho,
                col sep=comma
            ] {ppACcpu_time_seconds.csv};
            \addlegendentry{AC}

            \end{axis}
        \end{tikzpicture}
        \caption{Computation time (s)}
        \label{fig:rudy_cpu}
    \end{subfigure}
    \hfill
     \begin{subfigure}{0.49\textwidth}
        \centering
        \begin{tikzpicture}
            \begin{axis}[
                width=\linewidth,
                ylabel={Number of problems solved},
                %xmode=log,
                grid=major,
                legend pos=south east,
                legend style={font=\tiny},
                ymajorgrids=true,
                xmajorgrids=true,
            ]

            \addplot[
                thick,
                const plot
            ] table[
                x=tau,
                y=rho,
                col sep=comma
            ] {ppLSpgd_iterations.csv};
            \addlegendentry{LS}

            \addplot[
                red,
                thick,
                const plot
            ] table[
                x=tau,
                y=rho,
                col sep=comma
            ] {ppACpgd_iterations.csv};
            \addlegendentry{AC}

            \end{axis}
        \end{tikzpicture}
        \caption{PG iterations}
        \label{fig:rudy_pgd}
    \end{subfigure}
    \vspace{0.5cm}

    % ============================================================
    % Row 3: Newton iterations
    % ============================================================

    \begin{subfigure}{0.49\textwidth}
        \centering
        \begin{tikzpicture}
            \begin{axis}[
                width=\linewidth,
                ylabel={Number of problems solved},
                scaled x ticks=base 10:-3,
                %xmode=log,
                grid=major,
                legend pos=south east,
                legend style={font=\tiny},
                ymajorgrids=true,
                xmajorgrids=true,
            ]

            \addplot[
                thick,
                const plot
            ] table[
                x=tau,
                y=rho,
                col sep=comma
            ] {ppLSnewton_iterations.csv};
            \addlegendentry{LS}

            \addplot[
                red,
                thick,
                const plot
            ] table[
                x=tau,
                y=rho,
                col sep=comma
            ] {ppACnewton_iterations.csv};
            \addlegendentry{AC}

            \end{axis}
        \end{tikzpicture}
        \caption{Newton iterations}
        \label{fig:rudy_newton}
    \end{subfigure}
    \hfill

    \caption{Performance profiles comparing the projected subgradient method with linesearch and auto-conditioning on the MAXCUT problems.}
    \label{fig:performance_profiles}
\end{figure}
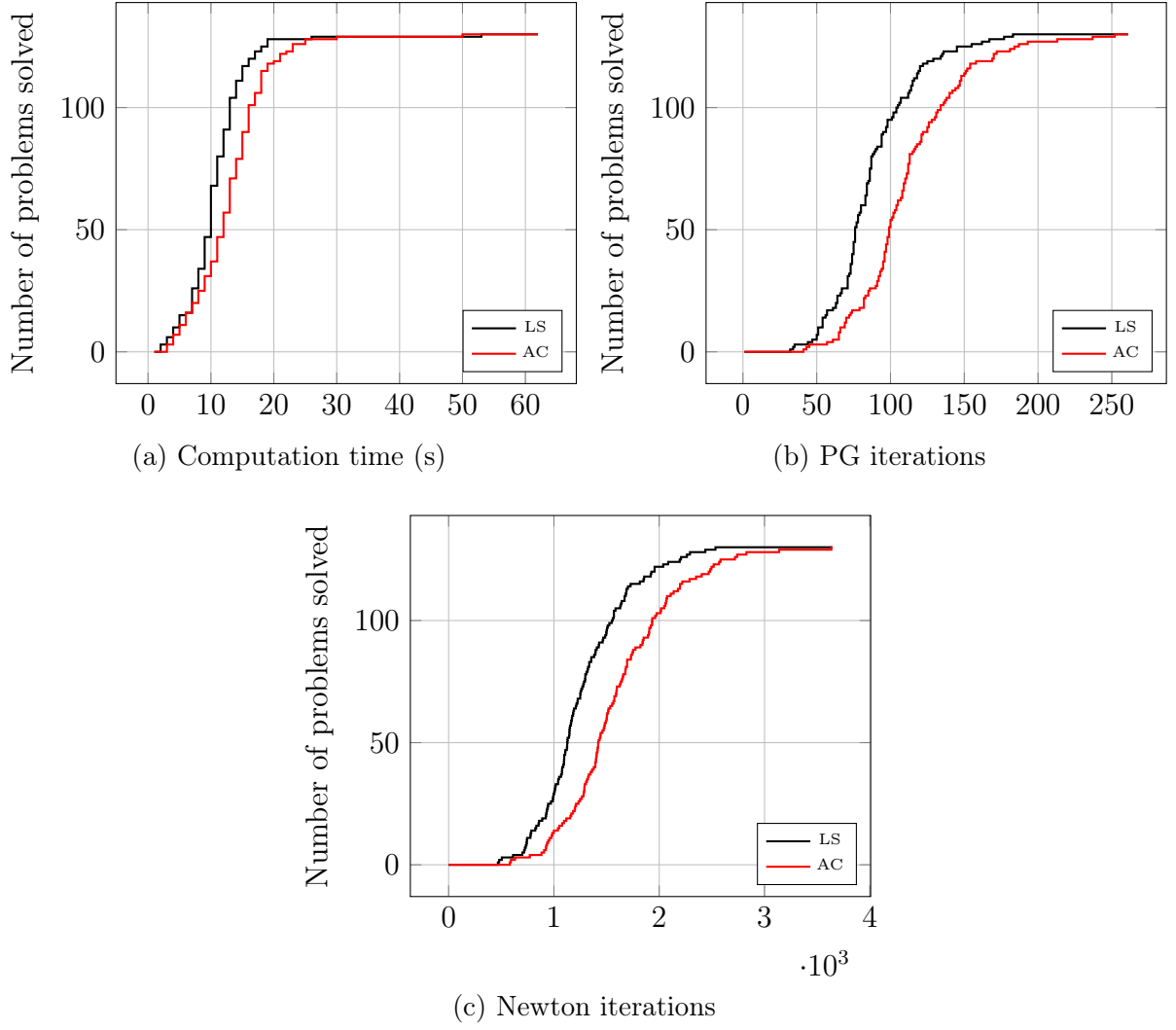

We note that the linesearch performs slightly better than the auto-conditioned variant across all metrics. This may be explained by the fact that we observed quick acceptance of the descent direction in most of the iterations of PG with linesearch, hence the auto-conditioned version cannot save significant computational effort compared to the linesearch in this case.

\subsection{Robust Principal Component Analysis}

Robust Principal Component Analysis (Robust PCA) was first introduced in \cite{RPCAWright} as the following task: Let $Y=L+S$ be a given matrix $Y\in\R^{m\times n}$ that is the sum of an unknown low-rank signal $L\in \R^{m\times n}$ and a sparse error term $S\in \R^{m\times n}$. We can only observe $Y$ and want to reconstruct $L$ and $S$. This leads to the following optimization problem
\begin{equation}\label{prob:rpca}
 \min_{L,S\in \R^{m\times n}}\|Y-L-S\|_F \quad \text{s.t.}\quad \rank L \leq r,\, \|S\|_0\leq k
\end{equation}
for $0\leq r \leq \min \{m,n\}$ and $0\leq k\leq mn$, where $\|S\|_0$ denotes the number of non-zero entries in $S$. The initial approach in \cite{RPCAWright} was to convexify the problem, while more recent approaches like the one in \cite{HammRPCA} tackle problem \eqref{prob:rpca} directly. We also want to follow the direct approach, but first have to bring problem \eqref{prob:rpca} into a suitable form for our algorithmic framework. First we define the sets
\begin{align*}
D:=\big\{A\in \R^{m\times n}\mid \rank A \leq r\big\}
\end{align*}
and
\begin{align*}
 M:=\big\{Y-S\mid \|S\|_0\leq k\big\}.
\end{align*}
Then we can reformulate problem \eqref{prob:rpca} as
\begin{equation}\label{prob:rpcaUpperC2}
 \min_{L\in \R^{m\times n}}\dist_M(L)^2 \quad\text{s.t.}\quad L\in D.
\end{equation}
This representation has the following properties: The squared distance function is upper-$\mathcal{C}^2$, as for arbitrary $X\in \R^{m\times n}$, it can be written as
\begin{align*}
 \dist_M(X)^2&=\inf_{Y\in M}\|X-Y\|^2 \\
 &=\inf_{Y\in M}\Big\{\|X\|^2-2\langle X,Y\rangle +\|Y\|^2\Big\} \\
 &=\|X\|^2 -\sup_{Y\in M}\Big\{2\langle X,Y\rangle - \|Y\|^2\Big\}.
\end{align*}
Setting
\begin{align*}
 g(X):=\|X\|^2, \quad h(X):=\sup_{Y\in M}\Big\{2\langle X,Y\rangle - \|Y\|^2\Big\},
\end{align*}
we note that $X\mapsto 2\langle X,Y\rangle - \|Y\|^2$ as an affine function is in particular convex for each $Y\in M$, and thus taking the supremum over all $Y\in M$ implies the convexity of $h$, cf.\ \cite[Proposition 8.16]{BauschkeCombettes}. Thus we can write $\dist_M^2 = g-h$, with $g$ differentiable with Lipschitz gradient and $h$ convex, and Proposition~\ref{Prop:CharUpperC2} (c) is again applicable.

The limiting subgradient of the squared distance function is well known, we have
\begin{align*}
 \partial \dist_M(X)^2=X-\proj_M(X)
\end{align*}
cf.\ \cite[Lemma 7]{GarrigosDistSquared}. Projecting a matrix $X\in \R^{m\times n}$ onto the matrices with rank less than $r$ is also a standard task and is done by truncating the singular value decomposition of $X$ to the largest $r$ singular values, and setting the rest to zero, cf.\ \cite[Section 6.2]{OlikierStratifiedSets}. In order to calculate the projection onto $M$, which is used in the calculation of $\dist_M^2$ as well as its subgradient, we note that by construction an element $W\in \proj_M(X)$ is obtained by taking the $k$ largest absolute value entries in $Y-X$ and setting the others to zero.

In order to test the robust PCA implementation, we used the dataset \texttt{CDnet2014}\footnote{\url{https://changedetection.net/dataset2014/}}\cite{CDnet2014}, which consists of different image sequences from videos, mostly from static surveillance camera footage. On this data, we can use robust PCA to do video background subtraction. We construct our observation matrix $Y\in \R^{m\times n}$ by setting its columns to a vector of length $m=pq$ containing the flattened version of each gray scale video frame of resolution $p\times q$. The number $n$ of columns in $Y$ is then the number of frames in the sequence we want to analyze. The intuition for using robust PCA to separate the video background from foreground is that in surveillance camera footage, the camera is static, hence the background components should have a constant pixel value across the whole observation sequence. Thus, if we only had the background in $Y$, we would have constant rows, and $Y$ would be a rank-$1$-matrix. However, the background is perturbed by the dynamic movements in the foreground, which can now be modeled as the sparse error term $S\in \R^{m\times n}$, as we also assume the foreground motion not to completely dominate the image. Hence, in this setting, the task of robust PCA is to determine the low rank background $L$, which can be used to separate background and foreground in our observation $Y$.

\begin{table}
 \centering
 \begin{tabular}{c|cc}
 Name &  PGD iterations & CPU time in seconds \\
 \hline
 \texttt{PETS2006} LS & $16$ & $109.56$ \\
 \texttt{PETS2006} AC & $17$ & $116.47$ \\
 \hline
 \texttt{highway} LS & $3$ & $20.23$ \\
 \texttt{highway} AC & $3$ & $19.87$ \\
 \hline
 \texttt{skating} LS & $9$ & $62.14$ \\
 \texttt{skating} AC & $8$ & $56.32$ \\
\end{tabular}
 \caption{Algorithmic Data for the robust PCA example}
  \label{tab:RPCAData}
\end{table}

\begin{figure}[htbp]
    \centering
    \setlength{\tabcolsep}{3pt}

    \begin{tabular}{c|c c c c}
        Name &
        V. &
        Original image $Y$ &
        Low-rank background $L$ &
        $|Y-L|$ (pointwise) \\[4pt]
        \hline

        % PETS2006
        \multirow{2}{*}{\texttt{PETS2006}} &
        LS &
        \includegraphics[width=0.25\textwidth]{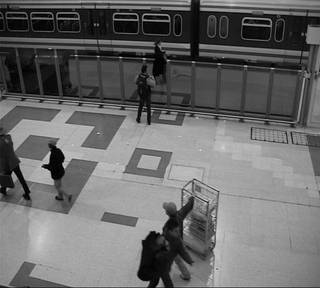} &
        \includegraphics[width=0.25\textwidth]{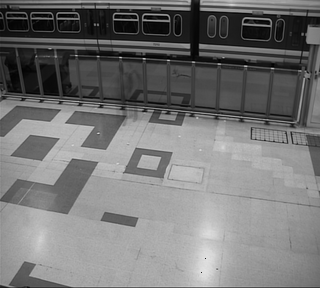} &
        \includegraphics[width=0.25\textwidth]{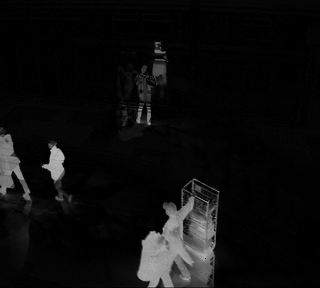}
        \\[4pt]

        &
        AC &
        \includegraphics[width=0.25\textwidth]{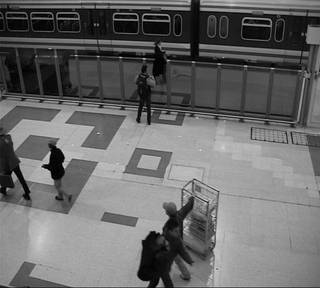} &
        \includegraphics[width=0.25\textwidth]{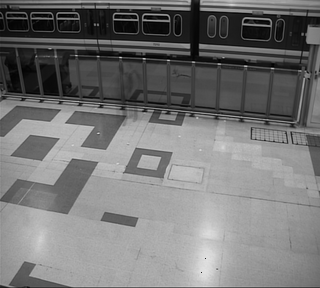} &
        \includegraphics[width=0.25\textwidth]{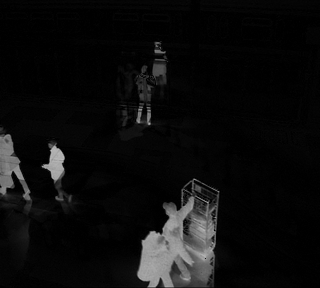}
        \\[8pt]

        \hline
        % highway
        \multirow{2}{*}{\texttt{highway}} &
        LS &
        \includegraphics[width=0.25\textwidth]{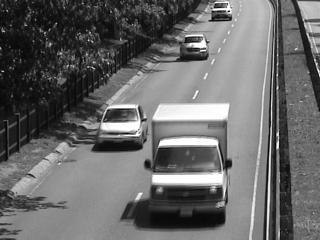} &
        \includegraphics[width=0.25\textwidth]{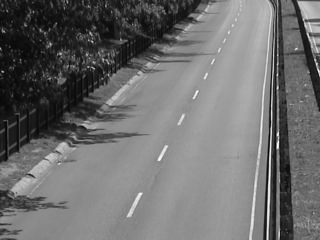} &
        \includegraphics[width=0.25\textwidth]{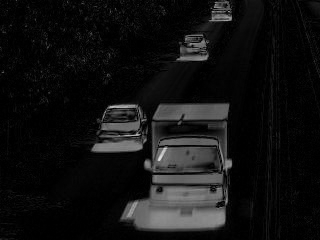}
        \\[4pt]

        &
        AC &
        \includegraphics[width=0.25\textwidth]{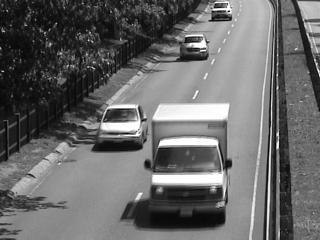} &
        \includegraphics[width=0.25\textwidth]{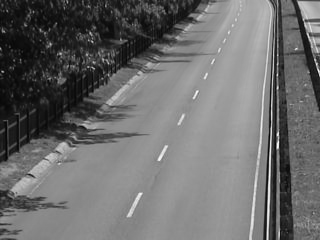} &
        \includegraphics[width=0.25\textwidth]{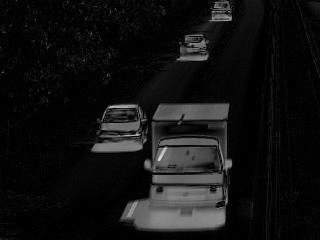}
        \\[8pt]

        \hline
        % skating
        \multirow{2}{*}{\texttt{skating}} &
        LS &
        \includegraphics[width=0.25\textwidth]{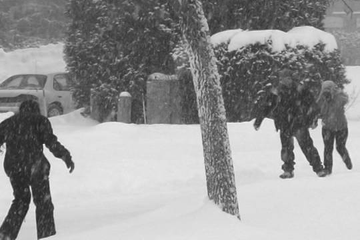} &
        \includegraphics[width=0.25\textwidth]{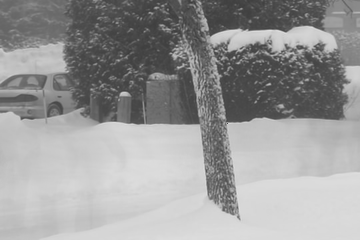} &
        \includegraphics[width=0.25\textwidth]{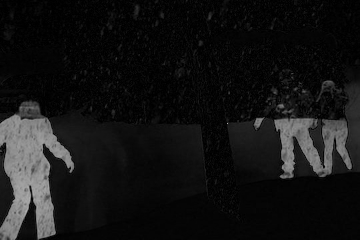}
        \\[4pt]

        &
        AC &
        \includegraphics[width=0.25\textwidth]{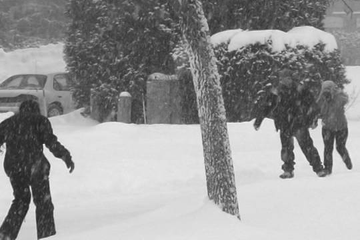} &
        \includegraphics[width=0.25\textwidth]{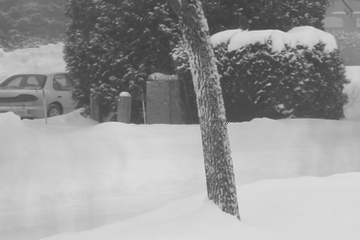} &
        \includegraphics[width=0.25\textwidth]{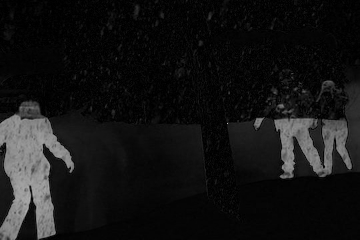}

    \end{tabular}

    \caption{RPCA results for video background subtraction.}
    \label{fig:rpcaImages}
\end{figure}

We ran the projected subgradient method in linesearch and auto-conditioned version on three different image sequences \texttt{PETS2006} ($288\times 320$), \texttt{highway} ($240\times 320$) and \texttt{skating} ($240\times 360$) of length $n=400$. Starting with a random matrix of the same dimensions as $Y$, all algorithms converged with desirable visual results displayed in Figure~\ref{fig:rpcaImages} for single example frames. Here we set $r=2$, $k=mn\cdot 10^{-4}$, $\tau = 10^{-5}$ and $\varepsilon=10^{-4}$. Computational data is displayed in Table~\ref{tab:RPCAData}. The computation times are overall comparable to the ones in \cite{AccAltProj} obtained by an accelerated alternating projection algorithm, but should be also read with caution, as they are highly dependent on the problem dimension, which in return scales with the underlying image resolution and sequence length.

\section{Final Remarks}\label{Sec:Final}

In this paper, we introduced the projected subgradient method for a class of nonsmooth objectives, namely upper-$\mathcal{C}^2$ functions. We showed that both presented algorithm versions, the one with a nonmonotone linesearch in Algorithm~\ref{Alg:NonmonotoneSubgradient} and the auto-conditioned Algorithm~\ref{Alg:ACPSubgradient}, posess strong convergence guarantees to a proximal critical point under mild assumptions. Furthermore, we validated the theoretical guarantees in a numerical setting and applied the method to solve real world numerical tasks with promising results. Possible directions for our future research include the design of new methods for the considered optimization problem with convergence to points which satisfy even stronger stationarity notions. Furthermore, the investigation of a generalization to proximal (sub-)gradient methods deserves attention. However, as already noted in \cite{olikier2025}, this would need significantly new techniques.
\par\addvspace{\baselineskip}

\noindent
\small\textbf{Funding.}
No funding was received to assist with the preparation of this manuscript.
\par\addvspace{\baselineskip}

\noindent
\small\textbf{Code and Data Availability.}
The source code for the implementation of the projected subgradient methods as well as the MPEC style, MAXCUT and Robust PCA numerical examples is available at \url{https://github.com/jakrueg/projected-subgradient}, the problem data can be sourced from \url{https://biqmac.aau.at/biqmaclib.html} and \url{https://changedetection.net/dataset2014/} respectively.
\par\addvspace{\baselineskip}

\noindent
\small\textbf{Conflict of interest.}
The authors have no competing interests to declare that are relevant to the content of
this article.
\par\addvspace{\baselineskip}

\printbibliography

\end{document}